%% file: SurveySymbolicPowers.tex
\documentclass[11pt]{amsart}
\usepackage[colorlinks=true,pagebackref,hyperindex,citecolor=green,linkcolor=red]{hyperref}
\usepackage{amsmath}
\usepackage{amsfonts}
\usepackage{amssymb}
\usepackage{tikz}
\usepackage{color}
\usepackage{graphicx}
\usepackage[all]{xy}  
\usepackage{enumerate}
\usepackage[top=1in, bottom=1in, left=0.9in, right=0.9in]{geometry}
\usepackage{mathrsfs}
\newcommand{\cb}{ }

\usepackage{stmaryrd}

\makeatletter
\def\namedlabel#1#2{\begingroup
#2%
\def\@currentlabel{#2}%
\phantomsection\label{#1}\endgroup
}
\makeatother

\makeatletter
\def\@tocline#1#2#3#4#5#6#7{\relax
  \ifnum #1>\c@tocdepth 
  \else
    \par \addpenalty\@secpenalty\addvspace{#2}%
    \begingroup \hyphenpenalty\@M
    \@ifempty{#4}{%
      \@tempdima\csname r@tocindent\number#1\endcsname\relax
    }{%
      \@tempdima#4\relax
    }%
    \parindent\z@ \leftskip#3\relax \advance\leftskip\@tempdima\relax
    \rightskip\@pnumwidth plus4em \parfillskip-\@pnumwidth
    #5\leavevmode\hskip-\@tempdima
      \ifcase #1
       \or\or \hskip 1.9em \or \hskip 2em \else \hskip 3em \fi%
      #6\nobreak\relax
    \dotfill\hbox to\@pnumwidth{\@tocpagenum{#7}}\par
    \nobreak
    \endgroup
  \fi}
\makeatother

\theoremstyle{definition}
\newtheorem{theorem}{Theorem}[section]

\newtheorem{hypothesis}[theorem]{Hypothesis}
\newtheorem{question}[theorem]{Question}
\newtheorem{corollary}[theorem]{Corollary}
\newtheorem{lemma}[theorem]{Lemma}
\newtheorem{exercise}[theorem]{Exercise}
\newtheorem{proposition}[theorem]{Proposition}

\newtheorem{notation}[theorem]{Notation}
\newtheorem{disc}[theorem]{Discussion}

\theoremstyle{definition}
\newtheorem{definition}[theorem]{Definition}

\newtheorem{example}[theorem]{Example}

\newtheorem{conjecture}[theorem]{Conjecture}
\newtheorem{remark}[theorem]{Remark}
\numberwithin{equation}{subsection}

\newcommand{\m}{\mathfrak{m}}
\newcommand{\n}{\mathfrak{n}}

\newcommand{\NN}{\mathbb{N}}

\newcommand{\CC}{\mathbb{C}}

\newcommand{\cV}{\mathcal{V}}

\newcommand{\cR}{\mathcal{R}}

\newcommand{\reg}{\operatorname{reg}}

\newcommand{\Spec}{\operatorname{Spec}}
\newcommand{\MaxSpec}{\operatorname{Max}}

\newcommand{\Hom}{\operatorname{Hom}}

\newcommand{\ini}{\operatorname{in}}
\newcommand{\Tor}{\operatorname{Tor}}

\newcommand{\Ass}{\operatorname{Ass}}

\newcommand{\Char}{\operatorname{char}}
\newcommand{\Min}{\operatorname{min}}

\newcommand{\Ht}{\operatorname{ht}}	
\newcommand{\e}{\operatorname{e}}	
\newcommand{\ord}{\operatorname{ord}}
	
\newcommand{\height}{\operatorname{height}}

\newcommand{\dsp}[1]{{\langle #1\rangle}}

\renewcommand{\!}[1]{{#1}}

\newcommand{\ls}{\leqslant}%
\newcommand{\gs}{\geqslant}
\newcommand{\ds}{\displaystyle}

\newcommand{\p}{\mathfrak{p}}
\newcommand{\q}{\mathfrak{q}}

\newcommand{\ps}[1]{\llbracket {#1} \rrbracket}

\newcommand{\J}{\mathcal J} 
\newcommand{\Jac}{\operatorname{Jac}}
\newcommand{\Der}{\operatorname{Der}}

\vspace{.6cm}

\title{Symbolic powers of ideals}
\author[Dao]{Hailong Dao${^1}$}
\address{Department of Mathematics, University of Kansas, Lawrence, KS 66045-7523,
USA}
\email{hdao@math.ku.edu}

\author[De Stefani]{Alessandro De Stefani}
\address{Department of Mathematics, Royal Institute of Technology (KTH), Stockholm, 100 44, Sweden}
\email{ad9fa@virginia.edu}

\author[Grifo]{Elo\'isa Grifo}
\address{Department of Mathematics, University of Virginia, Charlottesville, VA 22904-4135, USA}
\email{eloisa.grifo@virginia.edu}

\author[Huneke]{Craig Huneke${^2}$}
\address{Department of Mathematics, University of Virginia, Charlottesville, VA 22904-4135, USA}
\email{huneke@virginia.edu}

\author[N\'u\~nez-Betancourt]{Luis N\'u\~nez-Betancourt${^3}$}
\address{Centro de Investigaci\'on en Matem\'aticas, Guanajuato, GTO, Mexico.}
\email{luisnub@cimat.mx}

\thanks{{$^1$}The first author was partially supported by NSA Grant H98230-16-1-0012}

\thanks{{$^2$}The fourth author was partially supported by the NSF Grant 1460638}
\thanks{{$^3$}The fifth author was partially supported by the NSF Grant 1502282 }

\begin{document}
\maketitle

\begin{abstract}
We survey classical and recent results on symbolic powers of ideals. We focus on properties and problems of symbolic powers over regular rings, on the comparison of symbolic and regular powers, and on the combinatorics of the symbolic powers of monomial ideals. In addition, we present some new results on these aspects of the subject. 
\end{abstract}

\tableofcontents



\input{intro}
\input{RegularRings}

\input{UniformBounds}
\input{PackingProblem}

\vspace{.4cm}

\section*{Acknowledgments}
\vspace{.4cm}

We thank Jeff Mermin for many helpful conversations concerning the packing problem, and in particular
for discussions leading to  Remark \ref{easycases} and Corollary \ref{mermin}.
We thank Jonathan Monta\~{n}o, Andrew Conner, Jack Jeffries, and Robert Walker for helpful comments.
Part of this work was done when the second and fifth authors were at the University of Virginia. They wish to thank this institution for its hospitality.
Finally, the fifth author  thanks the organizing committee for the Brazil-Mexico 2nd meeting on Singularities in Salvador, Bahia, Brazil, where this project was initiated.


\bibliographystyle{alpha}

\bibliography{References}


\end{document}

%% file: intro.tex
\section{Introduction}

In this survey, we discuss different algebraic, geometric, and combinatorial aspects of the symbolic powers. Specifically, we focus on the properties and problems of symbolic powers over regular rings, on the comparison of symbolic and regular powers, and on the combinatorics of the symbolic powers of monomial ideals. 

Given an ideal $I$ in a Noetherian domain $R$, its $n$-th symbolic power is defined by 
$$
I^{(n)}=\bigcap_{\p\in \Ass_R(R/I)} (I^n R_\p\cap R).
$$
For many purposes, one can focus on symbolic powers of  prime ideals. If $\p$ is a prime ideal, then $\p^{(n)}$ is the $\p$-primary component of $\p^n$. 

Symbolic powers do not, in general, coincide with the ordinary powers.
From the definition it follows that $\p^n\subseteq \p^{(n)}$ for all $n$, but the converse may fail. The simplest such example can be constructed ``generically" by letting $\p = (x,y)$ in the hypersurface
defined by $x^n-yz= 0$. It is easy to see that in this example $y$ is in $\p^{(n)}$ but is not
even in $\p^2$.  It is a bit more subtle in a polynomial ring; however,  for the prime ideal $\p=(x^4-yz, y^2-xz, x^3y-z^2)\subseteq K[x,y,z]$, $\p^{(2)}\neq \p^2.$

The study and use of symbolic powers has a long history in commutative algebra. Krull's
famous proof of his principal ideal theorem uses them in an essential way.  Of course they
first arose after primary decompositions were proved for Noetherian rings. Zariski
used symbolic powers in his study of the analytic normality of algebraic varieties. 
Chevalley's famous lemma comparing topologies states that in a complete local domain
the symbolic powers topology of any prime is finer than the $\m$-adic topology.
A crucial step in the vanishing theorem on local cohomology of Hartshorne and Lichtenbaum uses
that for a prime $\p$ defining a curve in a complete local domain, the powers of $\p$ are
cofinal with the symbolic powers of $\p$. This important property of being cofinal
was further developed by Schenzel in the 1970s, and is a critical point for much of
this survey.  Irena Swanson proved an important refinement which showed that
when the symbolic power topology of a prime $\p$ is cofinal with the usual powers, there is a
{\it linear} relationship between the two: there exists a constant $h$ such that for
all $n$, $\p^{(hn)}\subseteq \p^n$. 

In case the base ring is a polynomial ring over a field, one may interpret the $n$-th symbolic power as the sheaf of all function germs over $X=\Spec(R)$ vanishing to order greater than or equal to $n$ at $Z=\cV(\p)$. Furthermore,  if  $X$ is a smooth variety over a perfect field, then
\begin{equation}\label{EqSymbPowerSmooth}
\p^{(n)}=\{f\in R \;|\; f\in \m^n \hbox{ for every closed point }\m\in Z \}
\end{equation}  
by the Zariski-Nagata Theorem \cite{ZariskiHolFunct,Nagata}. The closed points in Equation \ref{EqSymbPowerSmooth} can be taken only in the smooth locus of $Z$ \cite{EisenbudHochster}  (see Subsection \ref{SecZariskiNagata} for details). In this survey, we present a characteristic-free proof of the Zariski-Nagata Theorem, which is based on the work of Zariski. Our approach uses the general definition of differential operators given by Grothendieck  \cite{EGA}. As a consequence of this approach, we present a specific constant for a uniform version of the uniform Chevalley Theorem \cite{HKVFinExt} for direct summands of polynomial rings (see Theorem \ref{ThmZNDS} and Corollary \ref{CorZNDet}). We point out that this result can be seen as a weaker version of the Zariski-Nagata Theorem which holds for complete local domains.

Equation \ref{EqSymbPowerSmooth} is used together with Euler's Formula to deduce $\p^{(2)}\subseteq \m \p$ for a homogeneous prime ideal $\p$ in a polynomial ring over a characteristic zero field, where $\m$ is the maximal homogeneous ideal. This says that the minimal homogeneous generators of $\p$ cannot have order at least $2$ at all the closed points of the variety defined by $\p$. Eisenbud and Mazur \cite{EM} conjectured that the same would hold for local rings of equal characteristic zero. We devote Subsection \ref{SubsecEM} to this conjecture. As shown in \cite{EM}, the fact that the symbolic square of a prime does not contain any minimal generators has close connections with the notion of evolutions. In Subsection \ref{SubsecEM}, we define what an evolution is, and state what it means for it to be trivial. The existence of non-trivial evolutions for certain kinds of rings played a crucial role in Wiles's proof of Fermat's Last Theorem \cite{Wiles}.

An important property of symbolic powers over regular rings containing fields, which was surprising at the time, is the theorem of Ein, Lazarfeld, and Smith \cite{ELS} in the geometric case in characteristic zero, and Hochster and Huneke \cite{HHpowers} in general for $d$-dimensional regular rings containing a field, which establishes that

\begin{equation}\label{EqUniform}
\p^{(dn)}\subseteq \p^{n}.
\end{equation}

The fact that the constant $d$ is uniform, so independent  of $\p$, is remarkable.
 This theorem motivated the following question:

\begin{question}
Let $(R,\m,K)$ be a complete local domain. Does there exist a constant $c$, depending only on $R$, such that $\p^{(cn)}\subseteq \p^{n}$ for every prime ideal $\p$?
\end{question}

A positive answer to this question would establish that $\p$-adic and $\p$-symbolic typologies are uniformly equivalent (see \cite{HunekeRaicu} for a survey on uniformity). In Section \ref{SecUniform}, we discuss the case where this question has a positive answer.
In particular, we review the work of Huneke, Katz and Validashti on isolated singularities \cite{HKV} and finite extensions \cite{HKVFinExt}. We take advantage of this work to answer a question of Takagi about direct summands given by finite group actions (see Theorems \ref{UniformDirectSummand} and Corollary \ref{CorTakagi}).

In the final section of this survey we discuss the combinatorics that are encoded in the symbolic powers of monomial ideals. First, we discuss the work of  Minh and Trung \cite{TrungMinhCM}, and Varbaro \cite{Varbaro} which characterizes the Cohen-Macaulayness property of the symbolic powers in terms of the corresponding simplicial complex (see Theorem \ref{ThmCMMatroid}).

We also discuss the equality $I^{(n)}=I^{n}$ for square-free monomial ideals. This problem was related to a conjecture of Conforti and Cornu\'elos \cite{CC} on the max-flow and min-cut properties by Gitler, Valencia and  Villarreal \cite{GVV} and Gitler, Reyes, and Villarreal \cite{GRV}. The Conforti-Cornu\'elos Conjecture is known in the context of symbolic powers as the Packing Problem (see Subsection \ref{SubsectionPackProb}), and it is a central problem in this theory. 
In subsection \ref{SubsectionPackProb}, we give a relative version of the Packing Problem, and study $I^{(n)}=I^{n}$  separately for each $n$. We discuss the theorem that establishes that if $I_G$ is a monomial edge ideal  associated to a graph $G$, then  $I^{(n)}=I^{n}$ for every $n$  if and only if $G$ is bipartite.

We point out that the research and literature on this topic is vast, rich, and active. For this reason, we cannot cover all the aspects about symbolic powers in this survey. In particular,
we do not cover the very large amount of material concerning symbolic powers and fat points.
See \cite{BoH} or \cite{HH13} for material and references. 

{\cb
Although this article is largely expository, we present several results that are new, to the best of our knowledge.  These  include an example where radical ideals of extensions do not have uniform behavior (Example \ref{uniform powers counterexample}) and uniform bounds for direct summands of polynomial rings (Theorem \ref{ThmZNDS}, Corollary \ref{CorZNDet} and Theorem \ref{UniformDirectSummand}). Also in Theorem \ref{k-packed edge ideals}, we give an alternative proof for a characterization of $k$-packed edge ideals. In order to  distinguish the new results from those which are already known, we make a citation explicitly after the numbering. 
}

%% file: RegularRings.tex
\vspace{.4cm}

\section{Symbolic Powers on Regular Rings}

\subsection{Zariski-Nagata Theorem}
\label{SecZariskiNagata}

In this section, we prove the Zariski-Nagata Theorem for regular rings
\cite{Nagata,ZariskiHolFunct}. This fundamental theorem establishes that the $n$-th symbolic power of an irreducible variety, or a prime ideal, consists precisely of the set of elements whose order is at least $n$ in every closed point in the variety. Specifically,
$$
\p^{(n)} \,\,\, =\bigcap_{\substack{\m\in \MaxSpec(R) \\ \p\subseteq \m}} \m^n
$$
whenever $\p$ is a prime ideal in a polynomial ring over a field.
The first step in proving this fact is showing that $\p^{(n)}\subseteq \m^n$ holds in a regular local ring, even in mixed characteristic. The previous fact is also known as the Zariski-Nagata Theorem. We present a proof of this fact based on rings of differential operators over polynomial rings over any ground field. We point out that this method also works for power series rings.

\begin{definition}
Let $R$ be a finitely generated $K$-algebra.
The $K$-linear differential operators of $R$ of order $n$, $D^{n}_R\subseteq \Hom_K(R,R)$, are defined inductively as follows. The differential operators of order zero are $D^{0}_R = R\cong\Hom_R(R,R)$. We say that $\delta\in \Hom_K(R,R)$ is an operator of order less than or equal to $n$ if $[\delta,r]=\delta r-r\delta$
is an operator of order less than or equal to $n-1$ for all $r \in D^{0}_R$. The ring of $K$-linear differential operators is defined by $D_R=\displaystyle\bigcup_{n\in\NN}D^{n}_R$. If $R$ is clear from the context, we drop the subscript referring to the ring.
\end{definition}

\begin{definition}
Let $R$ be a finitely generated $K$-algebra.
Let $I$ be an ideal of $R$, and let $n$ be a positive integer.
We define the $n$-th $K$-linear differential power of $I$ by
$$
I^\dsp{n}=\{f\in R \, | \, \delta(f)\in I \hbox{ for all } \delta\in D^{n-1}_R\}.
$$
\end{definition}

\begin{remark}
Since $D^{n-1}_R\subseteq D^n_R,$ it follows that $I^\dsp{n+1}\subseteq I^\dsp{n}.$ If $I\subseteq J,$
then $I^\dsp{n}\subseteq J^\dsp{n}$ for every $n\in\NN$.
\end{remark}

In order to prove the Zariski-Nagata Theorem, we need some properties of the differential powers.

\begin{proposition}
Let $R$ be a finitely generated $K$-algebra.
Let $I$ be an ideal of $R$, and $n$ be a positive integer.
Then, $I^\dsp{n}$ is an ideal.
\end{proposition}
\begin{proof}
It is straightforward to verify that $f,g\in I^\dsp{n}$ implies that $f+g\in I^\dsp{n}.$ It then suffices to show that $rf\in I^\dsp{n}$ for $r\in R$ and $f\in I^\dsp{n}$. Let $\delta\in D^{n-1}.$
Then, $\delta(rf)=[\delta, r](f)+r\delta(f).$ Since $f\in I^\dsp{n}\subseteq I^\dsp{n-1},$ and $[\delta, r]\in D^{n-2}$, we have that 
$[\delta, r](f)\in I$. Since $f\in I^\dsp{n}$, $\delta(f)\in I.$
We conclude that $\delta(rf)\in I.$ Hence, $rf\in I^\dsp{n}.$
\end{proof}

\begin{proposition}\label{LemmaPowerCont}
Let $R$ be a finitely generated $K$-algebra.
Let $I$ be an ideal of $R$, and $n$ be a positive integer.
Then, $I^n\subseteq I^\dsp{n}$.
\end{proposition}
\begin{proof}
We proceed by induction on $n$.\\

\noindent $\underline{n=1:}$ In this case, $I=I^\dsp{n}$ because $D^0=R.$\\

\noindent$\underline{n\Longrightarrow n+1:}$ Let $f\in I$, $g\in I^{n}$ and $\delta\in D^{n}.$
Then, $\delta(fg)=[\delta, f](g)+f\delta(g).$ Since $[\delta, r]\in D^{n-1}$ and $g\in I^{n}\subseteq I^\dsp{n}$ by the induction hypothesis, we have that 
$[\delta, f](g)\in I.$ Then, $\delta(fg)\in I,$ and so, $fg\in I^\dsp{n+1}.$
Hence, $I^{n+1}\subseteq I^\dsp{n+1}.$
\end{proof}

\begin{proposition}\label{Prop DSP Primary}
Let $R$ be a finitely generated $K$-algebra.
Let $\p$ be a prime ideal of $R$, and $n$ be a positive integer.
Then,  $\p^{\langle n\rangle}$ is $\p$-primary.
\end{proposition}
\begin{proof}
Once more, we use induction on $n$.\\

\noindent $\underline{n=1:}$ In this case, $\p^\dsp{n}=\p$ is a prime ideal. \\

\noindent$\underline{n\Longrightarrow n+1:}$ 
Let $r\not\in \p$ and $f\in \p$ such that $rf\in \p^\dsp{n+1}.$
Let  $\delta\in D^{n}.$ 
Then, $\delta(rf)=[\delta, r](f)+r\delta(f)\in \p.$
Since $rf\in \p^\dsp{n+1}\subseteq \p^\dsp{n},$
we have that $f\in \p^\dsp{n}$ by the induction hypothesis.
Then, $[\delta, r](f)\in \p$ , because $[\delta, r]\in D^{n-1}$. We conclude that $r\delta(f)=\delta(rf)-[\delta, r](f)\in \p.$ Then,  $r\delta(f)\in \p$, and so, $\delta(f)\in \p$, because $\p$ is a prime ideal and $r\not\in \p.$
Hence, $f\in \p^\dsp{n+1}.$
\end{proof}

\begin{remark}\label{RemDiffPowerMax}
Let $K$ be a field, $R$ be either $K[x_1,\ldots,x_d]$ or $K[[x_1,\ldots,x_d]]$, and $\m=(x_1,\ldots,x_d).$
In this case, 
$$D^n_R=R\left\langle \frac{1}{\alpha_1 !}\frac{\partial^{\alpha_1}}{\partial x^{\alpha_1}_1}\cdots \frac{1}{\alpha_d !}\frac{\partial^{\alpha_d}}{\partial x^{\alpha_d}_d} \ \bigg| \ \alpha_1+\ldots+\alpha_d \ls n\right\rangle$$
If $f\not\in \m^n,$ then $f$ has a monomial of the form $x^{\alpha_1}_1\cdots x^{\alpha_d}_d$, with nonzero coefficient $\lambda \in K$, which is minimal among all monomials appearing in $f$ under the graded lexicographical order. Applying the differential operator 
$\frac{1}{\alpha_1 !}\frac{\partial^{\alpha_1}}{\partial x^{\alpha_1}_1}\cdots \frac{1}{\alpha_d !}\frac{\partial^{\alpha_d}}{\partial x^{\alpha_d}_d}$ 
maps $\lambda x^{\alpha_1}_1\cdots x^{\alpha_d}_d$ to the nonzero element $\lambda \in K$, and any other monomial appearing in $f$ either to a non constant monomial or to zero. Consequently, $f\not\in \m^{\dsp{n}}.$
Hence, $\m^{\dsp{n}}\subseteq \m^n.$ 
Since $\m^n\subseteq \m^{\dsp{n}}$ 
by Lemma \ref{LemmaPowerCont},
we conclude that $\m^{\dsp{n}} =  \m^n$.
\end{remark}

\begin{exercise}\label{ExercisePowerm}
Let $K$ be a field of characteristic zero, $R$ be either $K[x_1, \ldots,x_d]$ or $K[[x_1,\ldots,x_d]]$, and $\m=(x_1,\ldots,x_d).$ Then, $(\m^t)^{\dsp{n} }=\m^{n+t-1}$.
\end{exercise}

\begin{theorem}[{Zariski-Nagata Theorem for polynomial and power series rings \cite{ZariskiHolFunct}}]
Let $K$ be a field, $R$ be either $K[x_1,\ldots,x_d]$ or $K[[x_1,\ldots,x_d]]$, and $\m=(x_1,\ldots,x_d).$
Let $\p\subseteq \m$ be a prime ideal.
Then, for any positive integer $n$, we have $\p^{(n)}\subseteq \m^n.$
Furthermore, if $\Char(K)=0$ and $\p$ is a prime ideal such that $\p\subseteq \m^t,$ then $\p^{(n)}\subseteq \m^{n+t-1}.$
\end{theorem}
\begin{proof}
We have that $\m^{\dsp{n}}=\m^n$ by Remark \ref{RemDiffPowerMax}.
Then, $\p^{(n)}\subseteq \p^\dsp{n}\subseteq \m^n $ by Proposition \ref{Prop DSP Primary}.
The second claim follows from the fact that 
$\p^{(n)}\subseteq \p^\dsp{n}\subseteq (\m^t)^{\dsp{n}}=\m^{n+t-1}$ by Exercise \ref{ExercisePowerm}.
\end{proof}

In order to show a more general version of Zariski-Nagata, we use the Hilbert-Samuel multiplicity:

\begin{definition}
Let $(R,\m,K)$ be a $d$-dimensional local ring.
The Hilbert-Samuel multiplicity of $R$ is defined by
$$
\e(R)=\lim\limits_{n\to\infty}\frac{d! \, \lambda(R/\m^n)}{n^{d}}.
$$
\end{definition}

The Hilbert-Samuel multiplicity is an important invariant which detects and measures singularities. For instance, under suitable hypotheses, $\e(R)=1$ if and only if $R$ is a regular ring. If $R$ is a regular local ring and $f\in\m$, then $\e(R/fR)=\ord(f)$, where $\ord(f)=\max\{t\in\NN \mid f\in \m^t\}$.
Furthermore, under mild assumptions, we have that $\e(R_\p)\ls \e(R)$ for every $\p\in \Spec(R).$

\begin{theorem}[{General version of Zariski-Nagata \cite{ZariskiHolFunct,Nagata}}]\label{ThmGralZN}
Let $(R,\m,K)$ be a regular local ring, and  $\p\subseteq \m$ be a prime ideal.
Then, $\p^{(n)}\subseteq \m^n.$
\end{theorem}
\begin{proof}
For any element $f\in \m$, we have that
$$
\max\{t\in\NN \mid f\in \p^{(t)}\}=
\max\{t\in\NN \mid f\in \p^tR_\p \}=\e((R/fR)_\p)\ls \e(R/fR)
=\max\{t\in\NN \mid f\in \m^t\}.
$$
As a consequence, we obtain that $f\in\p^{(n)}$ implies that $f\in\m^n.$
\end{proof}

We now present Eisenbud's and Hochster's \cite{EisenbudHochster} proof  for 
Zariski's Main Lemma on Holomorphic Functions \cite{ZariskiHolFunct}. We point out that this result works in greater generality (cf. \cite{EisenbudHochster}).

\begin{theorem}[{Zariski \cite{ZariskiHolFunct}, Eisenbud-Hochster \cite{EisenbudHochster}}]\label{ThmZariskiHolFunct}
Let $K$ be a field and $R$ be a finitely generated regular $K$-algebra.
Then, 
$$
\p^{(n)} \,\,\, = \,\,\bigcap_{\substack{\m\in \MaxSpec(R) \\ \p\subseteq \m}} \m^n
$$
for every prime ideal $\p\subseteq R$.
\end{theorem}
\begin{proof}
The containment
$$\p^{(n)} \,\, \subseteq \bigcap_{\substack{\m\in \MaxSpec(R)  \\ \p\subseteq \m}} \m^n$$
follows from the Zariski-Nagata Theorem.
In order to prove the other containment, it suffices to show that
$$
\bigcap_{\substack{\m\in \MaxSpec(R) \\ \p\subseteq \m}}  \left(\m^n \cdot R/\p^{(n)}\right)=0.
$$

Let $M=R/\p^{(n)}$ and $M_i=\p^i\cdot R/\p^{(n)}.$ 
As an $R/\p$-module, the locus where $M_{i}/M_{i+1}$ is a free module is open.
In addition, the regular locus 
of $R/\p$ is open. Therefore, there exists  $f\not\in \p$ such that  $(M_{i}/M_{i+1})_f$ is a free $(R/\p)_f$-module and $(R/\p)_f$ is regular.
We note that
$$
\bigcap_{\substack{\m\in \MaxSpec(R) \\ \p\subseteq \m}}  \left(\m^n \cdot R/\p^{(n)}\right) \subseteq \bigcap_{\substack{\m\in \MaxSpec(R) \\ \p\subseteq \m}}  \left(\m^n \cdot R/\p^{(n)}\right)_f,
$$
because $\m^n \cdot R/\p^{(n)}\subseteq (\m^n \cdot R/\p^{(n)})_f$, as
$f\not\in \p$ and $\Ass_{R/\p^{(n)}}=\{\p\}.$ Then we can replace $R$ by $R_f$ and assume that $R/\p$ is regular and $M_{i}/M_{i+1}$ is a free $R/\p$-module.

We claim that 
\begin{equation}\label{EqLocalInt}
\m^nM\cap M_i\subseteq \m M_i
\end{equation}
for every maximal ideal $\m$ in $R$. If suffices to show this equality locally. If $\q\neq \m,$
then $(\m^n M\cap M_i)_\q=(M_i)_\q = (\m M_i)_\q$. Thus, it suffices to show our claim for  $\q = \m.$ Since $(R/\p)_\m$ and $R_\m$ are regular, there exists a regular sequence $x_1,\ldots, x_d\in R_\m$ such that $(x_1,\ldots, x_c)R_\m=\p R_\m$ and $(x_1,\ldots, x_d)R_\m=\m R_m.$
Furthermore, $\p^{(n)}R_\m=\p^n R_\m$.
Then, 
$$
(\m^nM\cap M_i)_\m= \left(\frac{(x_1,\ldots, x_d)^n}{(x_1,\ldots, x_c)^n} R_\m\right)\cap\left( \frac{(x_1,\ldots, x_c)^i}{(x_1,\ldots, x_c)^n}R_\m\right)\subseteq
\left( \frac{(x_1,\ldots, x_d)(x_1,\ldots, x_c)^i}{(x_1,\ldots, x_c)^n}R_\m\right),
$$
which proves our claim.

We note that 
\begin{equation}\label{IntersectionMi}
\bigcap_{\substack{\m\in \MaxSpec(R) \\ \p\subseteq \m}} \left(\m (M_i/M_{i+1})\right)=0,
\end{equation}
because $M_i/M_{i+1}$ is $R/\p$-free and $$\p=\bigcap_{\substack{\m\in \MaxSpec(R) \\ \p\subseteq \m}} \m,$$
as $R$ is a Jacobson ring.

We consider an element 
$$
v\in \bigcap_{\substack{\m\in \MaxSpec(R) \\ \p\subseteq \m}}  \left( \m^n \cdot R/\p^{(n)}\right).
$$
We want to show that $v=0.$
If $v\neq 0,$ then we pick the largest integer $i$ such that $v\in M_i.$
From previous considerations, we deduce that
\begin{align*}
M_i\bigcap \left(\bigcap_{\substack{\m\in \MaxSpec(R) \\ \p\subseteq \m}}  \left(\m^n  \cdot R/\p^{(n)}\right)\right)& = \bigcap_{\substack{\m\in \MaxSpec(R) \\ \p\subseteq \m}} \left(M_i\cap  (\m^n \cdot R/\p^{(n)})\right)\\
&\subseteq  \bigcap_{\substack{\m\in \MaxSpec(R) \\ \p\subseteq \m}} \m M_i\hbox{ by Equation \ref{EqLocalInt}};\\
&\subseteq M_{i+1}\hbox{ by Equation \ref{IntersectionMi}}.
\end{align*}
Then, $v\in M_{i+1},$ which contradicts our choice for $i$. Hence, $v=0$.

\end{proof}

We now give an exercise that is helpful for the next theorem.

\begin{exercise}\label{ExecDiffPowerInt}
Let $\{I_\alpha\}_{\alpha\in A}$ be an indexed family of ideals.
Then, 
$$
\bigcap_{\alpha\in A} I^\dsp{n}_{\alpha}=\left( \bigcap_{\alpha\in A} I_{\alpha}\right)^{\dsp{n}}
$$
for every positive integer $n$.
\end{exercise}

As a consequence of Theorem \ref{ThmZariskiHolFunct}, we can show that differential powers and symbolic powers are the same for polynomial rings over any perfect field. This is usually presented only for fields of characteristic zero (see for instance \cite[Theorem 3.14]{Eisenbud}). 

\begin{proposition}[Zariski-Nagata]\label{PropSymbPowersDiffPowers}
Let $R=K[x_1,\ldots,x_d]$ be a polynomial ring over $K$.
If $K$ is a perfect field and $\p\subseteq R$ a prime ideal, then
$$
\p^{(n)}=\p^\dsp{n}.
$$
\end{proposition}
\begin{proof}
Let $\m\subseteq R$ be a maximal ideal. Using the fact that $K$ is perfect and Remark \ref{RemDiffPowerMax}, one can show that $\m^\dsp{n}=\m^n$ for every positive integer $n$ by going to $R\otimes_K\overline{K}$.
Then,
\begin{align*}
\p^{(n)}&=
\bigcap_{\substack{\m\in \MaxSpec(R) \\ \p\subseteq \m}} \m^n\hbox{ by Theorem \ref{ThmZariskiHolFunct}};\\
&=\bigcap_{\substack{\m\in \MaxSpec(R) \\ \p\subseteq \m}} \m^\dsp{n} =\left(\bigcap_{\substack{\m\in \MaxSpec(R) \\ \p\subseteq \m}} \m\right)^\dsp{n}\hbox{ by Exercise \ref{ExecDiffPowerInt}};\\
&={\p}^\dsp{n} \hbox{ because }R/\p\hbox{ is a Jacobson ring.}\\
\end{align*}
\end{proof}

\begin{exercise}
Let $R=K[x_1,\ldots,x_d]$ be a polynomial ring over $K$, and
 $K$ be a perfect field and $I\subseteq R$ be a radical ideal. Prove that
$I^{(n)}=I^\dsp{n}.$ Show that this theorem does hold not if $I$ is not radical (hint: find an example where $\Char(K)$ is prime).
\end{exercise}

We can also characterize the symbolic powers in terms of join of ideals.
This characterization is used  to compute symbolic powers of prime ideals in polynomial rings. We start by recalling the definition of the join of two ideals.

\begin{definition}
Let $R=K[x_1,\ldots,x_d]$ be a polynomial ring over $K$.
For ideals $I, J\subseteq R$, we consider the ideals $I[\underline{y}]\subseteq K[y_1,\ldots,y_d]$ and 
 $J[\underline{z}]\subseteq K[z_1,\ldots,z_d]$, the ideals obtained from changing the variables in $I$ and $J$.
We define the join ideal of $I$ and $J$ by
$$
I*J=( I[\underline{y}],J[\underline{z}], x_1-y_1-z_1,\ldots, x_d-y_d-z_d)\bigcap K[x_1,\ldots,x_d].
$$
\end{definition}
Suppose that $K=\CC.$ Let $V$ and $W$ denote  the vanishing set of $I$ and $J$ in $\CC^n$ respectively.  
Then, the vanishing set of $I*J$ is the Zariski closure of  
$$
\bigcup_{v\in V, w\in W}<v,w>,
$$
where $<v,w>$ denotes the complex line that joins $v$ and $w$.

\begin{theorem}[{Sullivant \cite[Proposition 2.8]{Sullivant}}]
Let $K$ be a perfect field.
Let $R=K[x_1,\ldots,x_d]$ be a polynomial ring, and $\eta=(x_1,\ldots,x_n).$
Then,
$$
\p^{(n)}=\p*\eta^n
$$
for every prime ideal $\p\subseteq R$. 
\end{theorem}
\begin{proof}[Proof Sketch:]
Let $K=\overline{K}$ denote the algebraic closure of $K$.
We have that $\m*\eta^n=\m^n$ for every maximal ideal $\m\subseteq R\otimes_K \overline{K}$. 
Since $K$ is perfect, $\m\otimes_K \overline{K}$ is the intersection of maximal ideals in $R\otimes_K \overline{K}.$
Using this fact together with the faithful flatness of field extensions, we deduce that $\m*\eta^n=\m^n$ for every maximal ideal $\m\subseteq R$.
Then, we have
$$
\p*\eta^n  =\left( \bigcap_{\substack{\m\in \MaxSpec(R) \\ \p\subseteq \m}} \m\right)*\eta^n =  \bigcap_{\substack{\m\in \MaxSpec(R) \\ \p\subseteq \m}} \m * \eta^n=\bigcap_{\substack{\m\in \MaxSpec(R) \\ \p\subseteq \m}} \m^n=\p^{(n)},
$$
where the last equality follows from Theorem \ref{ThmZariskiHolFunct}.
\end{proof}

\begin{remark}
The previous theorem gives an algorithm to compute the symbolic powers of radical ideals in a polynomial ring over a perfect field.
\end{remark}

\subsection{Uniform Bounds}
\label{SubsecRegRingUniformBounds}
The following striking result by Ein, Lazarsfeld and Smith shows that it is possible to find a uniform constant$, c$, that guarantees that $\p^{(cn)}\subseteq \p^n$ for smooth varieties over $\CC$, as the following theorem makes explicit. We will revisit this theme of uniformity in Section \ref{SecUniform}.

\begin{theorem}[{Ein-Lazarsfeld-Smith \cite{ELS}}] \label{USP-Poly-CharZero}
If $\p$ is a prime ideal of codimension $h$ in the coordinate ring of a smooth algebraic variety over $\CC$, then $\p^{(hn)} \subseteq \p^n$ for all $n \gs 1$. 
\end{theorem}

Hochster and Huneke extended Ein-Lazarsfeld-Smithf Theorem to regular local rings containing a field, by reduction to characteristic $p >0$ methods, and using tight closure arguments.

\begin{theorem}[{Hochster-Huneke \cite{HHpowers}}]\label{USP-Poly-CharP}
Let $(R,\m)$ be a regular local ring containing a field, let $\p$ be a prime ideal, and let $h$ be the height of $\p$. Then $\p^{(hn)} \subseteq \p^n$ for all $n \gs 1$. 
\end{theorem}

\begin{proof} 
If $\p$ is a maximal ideal, then symbolic powers coincide with regular powers, and the statement is clear. If $\dim(R) \ls 1$, then either $\p=0$ or $\p$ is a maximal ideal, and the statement follows. Therefore, we assume that $\p$ is neither maximal nor zero. We first assume that $\Char(R) = p>0$. 
Fix $n \gs 1$ and let $f \in \p^{(hn)}$. For all $q=p^e$ write $q=an+r$ for some $a \in \NN$ and $0\ls r <n$. Then $f^a \in (\p^{(hn)})^a \subseteq \p^{(han)}$. Thus, $\p^{hn}f^a \subseteq \p^{hr} f^a \subseteq \p^{(han+hr))} = \p^{(hq)}$. We now want to show that $\p^{(hq)} \subseteq \p^{[q]}$ for all $q$. Since $\Ass_R(R/\p^{[q]}) = \{\p\}$ by the flatness of Frobenius, we can check the containment just after localizing at $\p$. Since $\p R_\p$ is the maximal ideal of the regular local ring $R_\p$, it is generated by $h$ elements. Furthermore, $\p^{(hq)}R_\p = \p^{hq}R_\p$. By the pigeonhole principle, $\p^{hq}R_\p \subseteq \p^{[q]}R_\p$, and this shows the containment in $R$ as well. Taking $n$-th powers yields $\p^{n^2h}f^{an} \subseteq (\p^n)^{[q]}$ for all $q$, and multiplying by $f^r$ finally yields $\p^{n^2h}f^q \subseteq (\p^n)^{[q]}$. Choose any non-zero $c \in \p^{n^2h}$, then $cf^q \in (\p^n)^{[q]}$. Therefore 
\[
\ds c \in \bigcap_{q} \left( (\p^n)^{[q]}:f^q\right) = \bigcap_q (\p^n:f)^{[q]},
\]
where the last equality follows from the flatness of Frobenius. Thus, either $f \in \p^n$, and we are done, or $\p^n:f \subseteq \m$, so that $c \in \bigcap_q \m^{[q]} \subseteq \bigcap_q \m^{q}= (0)$, which is a contradiction.
The result in characteristic zero follows by reduction to prime characteristic (see \cite{HHCharZero}).
\end{proof}

The previous two theorems can be restated for polynomial rings as follows.
If $\p$ is a prime ideal in $R=K[x_1,\ldots,x_d],$ then
$$
\bigcap_{\substack{\m\in \MaxSpec(R) \\ \p\subseteq \m}} \m^{dn} \subseteq 
\left(\bigcap_{\substack{\m\in \MaxSpec(R) \\ \p\subseteq \m}} \m\right)^n.
$$
This is a surprising fact, because the intersection is infinite.

The remarkable containment given by \ref{USP-Poly-CharP} and \ref{USP-Poly-CharZero} might not, however, be the best possible. Given a prime ideal $\p$ in a regular local ring and an integer $a$, one may ask what is the smallest $b \gs a$ such that $\p^{(b)} \subseteq \p^a$.

\begin{question}[Huneke]\label{23question}
Given a codimension $2$ prime ideal $\p$ in a regular local ring, does the containment
$$\p^{(3)} \subseteq \p^2$$
always hold?
\end{question}

Over the past decade, there has been a lot of work towards answering different versions of this question. If we consider the previous question  for a radical ideal, $I$, the containment of the third symbolic power in the square has been shown   to not hold in general \cite{counterexamples}, with an example later extended \cite{HaSe} \footnote{{\cb Akesseh \cite{Akesseh} and  Walker \cite{RobertHH,Robert17}  has also made recent progress regarding Question \ref{23question}.}}. 
However, the containment does hold when $I$ is a monomial ideal in a polynomial ring  \cite[8.4.5]{Seshadri}, or an ideal defining a set of general points in $\mathbb{P}^2$ \cite{HH13} and in $\mathbb{P}^3$ \cite{Dumnicki2015}. 

Harbourne has extended the question to higher powers \cite{HH13,Seshadri}:\footnote{{\cb The third and fourth authors \cite{GH} recently answered Question \ref{23question} affirmatively and proved Conjecture \ref{Harbourne} for ideals defining $F$-pure rings.}}
{\cb
\begin{conjecture}[Harbourne]\label{Harbourne}
	Given a radical homogeneous ideal $I$ in $k[\mathbb{P}^N]$, let $h$ be the maximal height of an associated prime of $I$. Then for all $n \geqslant 1$, 
$$I^{(hn-h+1)} \subseteq I^n.$$
\end{conjecture}
}
For a survey on the containment problem see \cite{ContainmentSurvey}.

Harbourne and Bocci introduced the resurgence of an ideal as an asymptotic measure of the best possible containment as the following definition makes explicit. 

\begin{definition}[{Harbourne-Bocci \cite{BoH,resurgence2}}]
Let $I\subseteq K[x_1,\ldots,x_n]$ be an homogeneous ideal.
The resurgence of $I$ is defined by
$$
\rho(I)=\sup\left\{ \frac{n}{m}\; |\;  I^{(n)} \not\subseteq I^m\right\}.
$$
\end{definition}

By Theorems \ref{USP-Poly-CharZero} and \ref{USP-Poly-CharP}, $\rho(I)\ls \dim(R)$.  However, computing the resurgence of an ideal might be a very difficult task -- instead, one may find bounds in terms of other invariants. One of these invariants is the Waldschmidt constant, which measures the asymptotic growth of the minimal degrees of the symbolic powers of the given ideal.

\begin{definition}[{Waldschmidt \cite{Waldschmidt}}]
Let $I\subseteq K[x_1,\ldots,x_n]$ be an homogeneous ideal, and 
$\alpha(I)=\min\{ t \;|\; I_t \neq 0\}$. The Waldschmidt constant of $I$ is then defined to be
$$
\widehat{\alpha}(I)=\lim\limits_{m\to\infty}\frac{\alpha(I^{(n)})}{n}.
$$
\end{definition}

We point out that Waldschmidt showed that the previous limit exits. Bocci and Harbourne have showed that  $\alpha(I)/\widehat{\alpha}(I) \ls \rho(I)$  \cite[Theorem 1.2]{BoH}. 
It is worth mentioning that the Zariski-Nagata Theorem (Theorem \ref{ThmGralZN}) guarantees that $1\ls \widehat{\alpha}(I)$.

There are several cases where the Waldschmidt constant has been computed \cite{GHV2013,Points2014,DHNSST2015,FHL2015,WaldschmidtMonomials} or bounded \cite{DHST2014,Dumnicki2015}.
{\cb We point out that the function $\reg(R/I^{(n)})$ has also been studied  \cite{Cutkosky,HLN02,LT10,HT16,MTRegSymb}. }

\subsection{Eisenbud-Mazur Conjecture}
\label{SubsecEM}
In this section we survey a famous conjecture of Eisenbud and Mazur, that can be stated in terms of containments involving symbolic powers. Given any ideal $I$ inside a ring $R$, we always have an inclusion $I^{(2)} \subseteq I$. However, it is natural to ask whether something more precise can be said about the containment. Note that, if $K$ is a field of characteristic zero, and $I \subseteq K[x_1,\ldots,x_n]$ is a homogeneous ideal, then for any homogeneous $f \in I^{(2)}$, say of degree $D$, we have
\[
\ds f = \frac{1}{D}\sum_{i=1}^n x_i \frac{\partial{f}}{\partial x_i} \in (x_1,\ldots,x_n) I,
\]
since $\partial(f)/\partial x_i \in I$ for all $i$ by Exercise \ref{ExecDiffPowerInt} and Proposition \ref{PropSymbPowersDiffPowers}. In other words, $f$ is never a minimal generator of the ideal $I$, whenever $f \in I^{(2)}$. One can wonder whether this is true more generally. 
\begin{conjecture}[Eisenbud-Mazur \cite{EM}] \label{EisMaz_conj} Let $(R,\m)$ be a localization of a polynomial ring $S = K[x_1,\ldots,x_n]$ over a field $K$ of characteristic zero. If $I \subseteq R$ is a radical ideal, then $I^{(2)} \subseteq \m I$. 
\end{conjecture}
Conjecture \ref{EisMaz_conj} easily fails if the ambient ring is not regular. For example, if $R=\CC[x,y,z]/(xy-z^2)$ and $I=(x,z)$, then $x \in I^{(2)} \smallsetminus \m I$. The assumption on the characteristic is also needed. In fact, E. Kunz provided a counterexample to Conjecture \ref{EisMaz_conj} for any prime integer $p$:

\begin{example} [\cite{EM}] Let $p$ be a prime integer, and consider the polynomial $f= x_1^{p+1}x_2 - x_2^{p+1}-x_1x_3^p+x_4^p \in \mathbb{F}_p[x_1,x_2,x_3,x_4]$. Note that $f$ is a quasi-homogeneous polynomial, and $f \notin (x_1,x_2,x_3,x_4)\sqrt{\Jac(f)}$. Let $I$ be the kernel of the map
\[
\xymatrixcolsep{5mm}
\xymatrixrowsep{0.2mm}
\xymatrix{
\mathbb{F}_p[x_1,x_2,x_3,x_4] \ar[rrr] &&& \mathbb{F}_p[t] \\
x_1 \ar@{|->}[rrr] &&&  t^{p^2} \\
x_2 \ar@{|->}[rrr] &&& t^{p(p+1)} \\
x_3 \ar@{|->}[rrr] &&& t^{p^2+p+1} \\
x_4 \ar@{|->}[rrr] &&& t^{(p+1)^2}.
}
\]
Then, $f \in I^{(2)} \smallsetminus (x_1,x_2,x_3,x_4)I.$
\end{example} 
Conjecture \ref{EisMaz_conj} is open in most cases when the base field $K$ has characteristic $0$. The most recent results in this direction, to the best of our knowledge, are due to A. A. More, who proves Conjecture \ref{EisMaz_conj} for certain primes in power series rings \cite{AAMore}.

Conjecture \ref{EisMaz_conj} is related to and motivated by the existence of non-trivial evolutions. In order to explain this connection, we first recall some basic facts about derivations and modules of K{\"a}hler differentials. For a more exhaustive and detailed treatment we refer the reader to \cite{Kunz_diff}.

Let $K$ be a Noetherian ring, and let $R$ be a $K$-algebra, essentially of finite type over $K$, and let $M$ be a finitely generated $R$-module. A $K$-derivation $\partial:R \to M$ is a $K$-linear map that satisfies the Leibniz rule:
\[
\ds \partial(rs) = \partial(r)s +  r \partial(s)
\]
for all $r,s \in R$. The set of all derivations $\Der_K(R,M)$ is an $R$-module. When $M=R$, we denote $\Der_K(R):=\Der_K(R,R)$. As in the previous section, recall that $D^n_R$ denotes the set of $K$-linear differential operators of $R$ of order at most $n$.
\begin{lemma}\label{der_diff} Every element $\delta \in D^1_R$ can be written as the sum of a derivation $\partial \in \Der_K(R)$ and an operator $\mu_r \in D_R^0$, that is, multiplication by some element $r \in R$.
\end{lemma}
\begin{proof}
Note that multiplication by elements of $R$ and derivations are differential operators of order zero and one, respectively. Now let $\delta \in D^1_R$. Let $\partial := \delta - \mu_{\delta(1)}$, where $\mu_{\delta(1)}(r) = \delta(1)r$ for all $r \in R$. Then $\partial$ is still a differential operator of order at most one, and it is clear that $\partial(\lambda) = 0$ for all $\lambda \in K$, by $K$-linearity of $\delta$. Since $\partial$ has order one, for all $r,s \in R$ we have
\[
\ds \mu_t (s) = [\partial,r](s) = \partial(rs) - r \partial(s)
\]
where $\mu_t$ is multiplication by some element $t \in R$. Applying this identity to $s=1$, using that $\partial(1)= 0$, we obtain that $t= \partial(r)$. Therefore
\[
\ds \partial(rs) = \partial(r)s + r \partial(s)
\]
for all $r,s \in R$, proving that $\partial \in \Der_K(S)$.
\end{proof}
Consider now the multiplication map $R \otimes_K R \to R$, and let $\mathcal{I}$ be its kernel. $\mathcal{I}$ is generated, both as a left and right $R$-module, by elements of the form $x \otimes 1 - 1 \otimes x$. In addition, one can show that $r (x \otimes 1 - 1 \otimes x) + \mathcal{I}^2 = (x \otimes 1 - 1 \otimes x) r + \mathcal{I}^2$, for all $r,x \in R$. We define
\[
\ds \Omega_{R/K}:= \mathcal{I}/\mathcal{I}^2
\]
which is an $R$-module (the actions on the left and on the right are the same, given the previous comment). The module of differentials comes equipped with a universal derivation $d_{R/K}:R \to \Omega_{R/K}$, which is the map that sends $r \in R$ to $r \otimes 1 - 1 \otimes r$. For any $R$-module $M$, we have an isomorphism
\[
\Der_K(R,M) \cong \Hom_R(\Omega_{R/K},M).
\]
In fact, every derivation $\partial \in \Der_K(R,M)$ can be written as $\partial = \varphi \ \circ \ d_{R/K}$, for some $R$-linear homomorphism $\varphi \in \Hom_R(\Omega_{R/K},M)$. Conversely, $\psi \circ d_{R/K}: R \to M$ is a $K$-derivation for any $\psi \in \Hom_R(\Omega_{R/K},M)$. We now recall how to explicitly describe $\Omega_{R/K}$ when we have a presentation of $R$ over $K$.

\begin{itemize}
\item If $R=K[x_1,\ldots,x_n]$ is a polynomial ring over $K$, then one can show that
\[
\ds \Omega_{R/K} \cong R dx_1 \oplus \ldots \oplus R dx_n
\]
is a free $R$-module of rank $n$, with basis labeled by symbols $dx_i$. In this case the universal derivation $d_{R/K}:R \to \Omega_{R/K}$ turns out to be the standard differential. More explicitly, for $f \in K[x_1,\ldots,x_n]$, we have
\[
\ds d_{R/K}(f) = \sum_{i=1}^n \frac{\partial(f)}{\partial x_i} dx_i.
\]
\item If $R=S/I$, where $S=K[x_1,\ldots,x_n]$ is a polynomial ring and $I \subseteq S$ is an ideal, then
\[
\ds \Omega_{R/K} \cong \frac{\Omega_{S/K}}{I \Omega_{S/K} + S \cdot d_{S/K}(I)},
\]
where $d_{S/K}(I) = \{d_{S/K}(f) \mid f \in I\}$. The universal derivation is the map $d_{R/K}: R \to \Omega_{R/K}$ induced by $d_{S/K}:S \to \Omega_{S/K}$ on $R$. 
\item If $R = T_W$, where $T=K[x_1,\ldots,x_n]/I$ and $W$ is a multiplicatively closed set, we have
\[
\ds \Omega_{R/K} \cong \left(\Omega_{T/K}\right)_W
\]
and $d_{R/K}$ obeys the classical quotient rule. 
\end{itemize}
These rules are helpful to compute the module of differentials $\Omega_{R/K}$ when $R$ is essentially of finite type over $K$. Given ring homomorphisms $K \to T \to R:=T/I$, we obtain an exact sequence
\begin{equation} \label{exact_seq_derivations}
\xymatrixcolsep{5mm}
\xymatrixrowsep{2mm}
\xymatrix{
\ds I/I^2 \ar[rr]^-{\alpha} && \Omega_{T/K} \otimes_T R \ar[rr]^-\beta && \Omega_{R/K} \ar[rr] &&  0
}
\end{equation}
where $\alpha(i+I^2) = d_{T/K}(i)+ I \Omega_{T/K}$ and $\beta(d_{T/K}(t) \otimes r) = d_{R/K}(t) \cdot r$. Here we are using the same notation for elements in a module and classes in quotients of the same module.

We are now finally in a position to define evolutions.
\begin{definition} 
Let $R$ be a local $K$-algebra, where $K$ is a field. An evolution of $R$ is a surjective homomorphism $T \to R$ of $K$-algebras such that $R \otimes_T \Omega_{T/K} \to \Omega_{R/K}$ is an isomorphism. The evolution is said to be trivial if $T \to R$ is an isomorphism, and $R$ is said to be evolutionary stable if every evolution of $R$ is trivial.
\end{definition}

Evolutions appear in the study of Hecke algebras and in the work of Wiles on Galois deformations, in relation with his proof of Fermat's Last Theorem. They have also been studied by Scheja-Storch \cite{SchejaStorch} and B{\"o}ger \cite{Boger}, under slightly different perspectives. Evolutions were formally introduced by Mazur \cite{Mazur} in relation with the work of Wiles on semistable curves \cite{Wiles}. This comes from a desire to compare some universal deformation ring with a particular quotient of it, which arises as a completion of a Hecke algebra. In many cases the induced quotient map is an evolution, and it was crucial to establish that it is  trivial. When introducing evolutions, Mazur asked whether any ring arising in such a way is actually evolutionary stable. In this direction, further work of  Wiles and Taylor-Wiles \cite{TaylorWiles} showed that any such evolution is trivial. However, Mazur's more general question whether reduced algebras essentially of finite over fields of characteristic zero are evolutionary stable is still open. Some partial results have been established  \cite{EM, Hubl, Hubl_monomial, HunekeRibbe, HublHuneke}. 

We start off by justifying the relation between evolutions and the Eisenbud-Mazur Conjecture \ref{EisMaz_conj}. We closely follow the original arguments by Eisenbud and Mazur \cite{EM}.

We first present a very useful criterion for existence of non-trivial evolutions, in terms of minimality, which is due to H. Lenstra. 

\begin{definition} 
Let $T$ be a ring, and let $\phi:M \to N$ be an epimorphism of $T$-modules. We say that $\phi$ is minimal if there is no proper submodule $M' \subseteq M$ such that $\phi(M') = N$.
\end{definition}

The following proposition actually works for local algebras essentially of finite type over any Noetherian ring. However, we just focus on algebras essentially of finite type over a field. 

\begin{proposition}[{Lenstra \cite[Proposition 1]{EM}}] \label{PropositionLenstra} Let $R$ be a local $K$-algebra, essentially of finite type over $K$. Then $R$ is evolutionary stable if and only if for some (equivalently all) presentations $R = S/I$, where $S$ is a localization of a polynomial ring over $K$, the map
\[
\ds \widetilde{\alpha}: I/I^2 \to \ker\left(R \otimes_S \Omega_{S/K} \to \Omega_{R/K}\right) \to 0
\]
induced from the exact sequence (\ref{exact_seq_derivations}) is minimal.
\end{proposition}
\begin{proof} We leave it to the reader to show that $\widetilde{\alpha}$ being minimal or not is independent of the chosen presentation. Let $S/I$ be any presentation of $R$, with $S$ a localization of a polynomial ring in finitely many variables over $K$. Let $J \subseteq I$ be an ideal, so that we have a surjection $A:=S/J \to S/I = R \to 0$. A diagram chase on the sequences (\ref{exact_seq_derivations}) induced by the surjections $S \to A \to R \to 0$ shows that $A\to R$ is an evolution if and only if the surjective map $\widetilde{\alpha} : I/I^2 \to \ker\left(\Omega_{S/K} \otimes_S R\right)$ carries $(J + I^ 2)/I^2$ onto the same image as $I /I^2$. In addition, by Nakayama's Lemma we have that $J = I$ if and only if $(J + I^2)/I^2 = I/I^2$. Therefore $R$ has no non-trivial evolutions of the form $S/J$ if and only if no proper submodule $(J+I^2)/I^2 \subsetneq  I/I^2$ has the same image as $I/I^2$ via $\widetilde{\alpha}$, if and only if $\widetilde{\alpha}$ is minimal.

\end{proof}

Under certain assumptions, we can explicitly identify the kernel of the map $\alpha$ considered above.

\begin{theorem}[{Eisenbud-Mazur \cite[Theorem 3]{EM}}] \label{Theorem_Kernel_alpha}
Let $(S, \m)$ be a localization of a polynomial ring in finitely many variables over $K$, and let $I$ be an ideal of $S$. If $R := S/I$ is reduced and generically separable over $K$, then the kernel of $\alpha \!: I/I^2 \to R \otimes_K \Omega_{S/K}$ is $I^{(2)}/I^2$. 
\end{theorem}

As a consequence of Theorem \ref{Theorem_Kernel_alpha}, we can now fully justify the connection between Conjecture \ref{EisMaz_conj} and the existence of non-trivial evolutions.

\begin{corollary} \label{EisMaz_conj_coroll} Let $(S,\m)$ be a localization of a Noetherian polynomial ring over a field $K$, and let $I$ be a radical ideal of $S$. If $R=S/I$ is generically separable over $K$, then $R$ is evolutionary stable if and only if $I^{(2)} \subseteq \m I$.
\end{corollary}
\begin{proof}
It is enough to observe that minimality of a surjective map $f \!:A \to B$ is equivalent to the fact that $\ker(f)$ does not contain any minimal generator of $A$.
\end{proof}

Over the complex numbers, Conjecture \ref{EisMaz_conj} can be equivalently restated even more explicitly.
\begin{proposition}[{Eisenbud-Mazur \cite[Corollary 2]{EM}}]
There exists a reduced local $\CC$-algebra $R$ of finite type whose localization at the origin is {\it not} evolutionary stable if and only if there exists a power series $f \in \CC\ps{x_1,\ldots,x_n}$ without a constant term such that
\[
\ds f \notin (x_1,\ldots,x_n)\sqrt{ (f,\partial_1(f),\ldots ,\partial_n(f))}.
\]
\end{proposition}

We now present some results due to H{\"u}bl \cite{Hubl}, that  lead to new versions of the Eisenbud-Mazur conjecture. We closely follow his treatment of these topics.

\begin{theorem}[{H{\"u}bl  \cite[Theorem 1.1]{Hubl}}] \label{Hubl1} Let $K$ be a Noetherian ring, and let $R$ be a local algebra essentially of finite type over $K$. The following conditions are equivalent:
\begin{enumerate}
\item \label{Hubl1,1} $R$ is evolutionary stable.
\item \label{Hubl1,2} Assume $(S,\m)$ is a local algebra, essentially of finite type and smooth over $K$, and $I \subseteq S$ is such $R=S/I$. If $f \in I$ and $\partial(f) \in I$ for all $\partial \in \Der_K(S)$, then $f \in \m I$.
\end{enumerate}
\end{theorem}
\begin{proof}
Assume (\ref{Hubl1,1}), and write $R=S/I$ for some radical ideal $I$. Let $f \in I$ be such that $\partial(f) \in I$ for all $\partial \in \Der_K(S)$. Assume, by way of contradiction, that $f \notin \m I$. It follows that we can find $J \subseteq I$ such that $I = J+(f)$, so that $I/J \cong S/\m$. Let $T:=S/J$, and consider the surjection $T \to R$; we claim that this is an evolution, and this  gives a contradiction. In fact, if we let $\mathcal{I}:=I/J \subseteq T$, we have an exact sequence
\[
\xymatrixcolsep{5mm}
\xymatrixrowsep{2mm}
\xymatrix{
\mathcal{I}/\mathcal{I}^2 \cong (f + J)/J \ar[r]^-\alpha & \Omega_{S/K}/ \mathcal{I} \Omega_{S/K} \cong \Omega_{T/K} \otimes_T R \ar[r] & \Omega_{S/K} \ar[r] & 0.
}
\]
If $\widetilde{f}$ is the image of $f$ inside $\mathcal{I}/\mathcal{I}^2$, we have that $\alpha(\widetilde{f}) = df + \mathcal{I}\Omega_{S/K}$. Note that $\Omega_{S/K}$ is free over $S$, because $S$ is essentially of finite type and smooth over $K$. Since $\partial(f) \in I$ for all $\partial \in \Der_K(S)$, and $\Der_K(S) \cong \Hom_S(\Omega_{S/K},S)$, we have that $d(f) \in I \Omega_{S/K}$. We then have that $\alpha(\widetilde{f}) = 0$ in $\Omega_{S/K}/\mathcal{I}\Omega_{S/K}$, which shows that $T \to R$ is an evolution.

For the converse, assume (\ref{Hubl1,2}), and consider any evolution $T=S/J \to S/I = R$ of $R$. By way of contradiction, assume that the evolution is non-trivial, so that $J \subsetneq I$. Without loss of generality, we can assume that $I = J+(f)$, for some $f \notin I$ such that $f \m \subseteq J$. Clearly, $f \notin \m I$, otherwise $I = J+(f) \subseteq J+ \m I \subseteq J$, contradicting the non-triviality of the evolution. We want to show that there exists $h \in I \smallsetminus \m I$ such that $\partial(h) \in I$ for all $\partial \in \Der_K(S)$. To find such $h$, note that
\[
\ds \Omega_{R/K} \cong \frac{\Omega_{S/K}}{I \Omega_{S/K} + S \cdot d(I)} = \frac{\Omega_{S/K}}{J \Omega_{S/K} + f \Omega_{S/K} + S \cdot d(I)}.
\]
On the other hand, since $T \to R$ is an evolution, we have
\[
\ds \Omega_{R/K} \cong \frac{\Omega_{T/K}}{f \Omega_{T/K}} \cong \frac{\Omega_{S/K}}{J \Omega_{S/K} + S \cdot d(J) + f \Omega_{S/K}}.
\]
This shows that
\[
\ds d(f) \in d(I) \subseteq J \Omega_{S/K} + S \cdot d(J) + f \Omega_{S/K} = I \Omega_{S/K} + S \cdot d(J).
\]
Let $g_1,\ldots,g_r \in J$ be such that $d(f) - \sum_i s_i d(g_i) = \eta \in I \Omega_{S/K}$, where $s_i \in S$. Set $h:= f-\sum_i s_ig_i$, and note that
\begin{itemize}
\item $d(h) = d(f) - \sum_i s_id(g_i) - \sum_i d(s_i)g_i = \eta - \sum_i d(s_i)g_i \in I \Omega_{S/K}$
\item $I = J + (f) = J + (h)$; therefore, $h \in I$. In addition, $h \notin \m I$; otherwise, $f \in \m I$.
\end{itemize}
Since $d(h) \in I \Omega_{S/K}$, and $\Der_K(S) \cong \Hom_S(\Omega_{S/K},S)$, we have that $\partial(h) \in I$ for all $\partial \in \Der_K(S)$. This concludes the proof.
\end{proof}
In light of Exercise \ref{ExecDiffPowerInt} and Proposition \ref{PropSymbPowersDiffPowers}, we see that, under the assumptions that $S=K[x_1,\ldots,x_n]$, $I \subseteq S$ is radical and $K$ is perfect, we have that $I^{(2)} = I^{\dsp{2}}$. Therefore Theorem \ref{Hubl1} becomes just a restatement of Corollary \ref{EisMaz_conj_coroll}.

\begin{theorem}[{H{\"u}bl \cite[Theorem 1.2]{Hubl}}]
 \label{Hubl2} 
 Let $K$ be a field of characteristic zero, and let $S$ be a smooth algebra, essentially of finite type over $K$. For a radical ideal $I \subseteq S$ and $f \in I$, the following conditions are equivalent:
\begin{enumerate}
\item \label{Hubl2,1} $\partial(f) \in I$ for all $\partial \in \Der_K(S)$;
\item \label{Hubl2,2} $f \in I^{(2)}$;
\item \label{Hubl2,3} $f^n \in I^{n+1}$ for some $n \in \NN$.
\end{enumerate}
\end{theorem}
\begin{proof} For simplicity, we prove only the case when $K$ is algebraically closed, and $S=K[x_1,\ldots,x_n]$.

Assume (\ref{Hubl2,1}). Let $\delta \in \mathcal{D}^1_K(S)$ be a differential operator of order at most one. By Lemma \ref{der_diff} we have that
\[
\ds \delta(f) = \partial(f) + \mu_{\delta(1)} (f) = \partial(f) + \delta(1) \cdot f \in I
\]
for all $\delta \in \mathcal{D}^1_K(S)$. Therefore, we obtain that $f \in I^{\langle 2 \rangle}$. Given that $K$ is a field of characteristic zero, and that $S$ is a localization of a finite algebra over $K$, we obtain that $f \in I^{(2)}$ by  Exercise \ref{ExecDiffPowerInt} and Proposition \ref{PropSymbPowersDiffPowers}, as desired.

Conversely, if $f \in I^{(2)} = I^{\langle 2 \rangle}$ we have that $\delta(f) \in I$ for all $\delta \in \mathcal{D}^2_K(S)$. Let $\partial \in \Der_K(S)$ be a derivation; then $\partial$ is, in particular, an element of $\mathcal{D}^2_K(S)$. In particular, $\partial(f) \in I$. Thus (\ref{Hubl2,1}) and (\ref{Hubl2,2}) are equivalent.

Now assume (\ref{Hubl2,1}). Let $Q$ be the field of fractions of $S$, and let $v_1,\ldots,v_t$ be the Rees valuations of $I$, with associated valuation rings $S \subseteq V_i \subseteq Q$. Note that $V_i$ is essentially of finite type over $S$. Since $\partial(f) \in I$ for all $\partial \in \Der_K(S)$, and $\Omega_{S/K}$ is free over $S$ by smoothness, we have that $d_{S/K}(f) \in I \Omega_{S/K}$. As the canonical map $\Omega_{S/K} \to \Omega_{V_i/K}$ is $S$-linear, we have that $d_{V_i/K}(f) \in I \Omega_{V_i/K}$. Similarly, we have that $d_{\widehat{V_i}/K}(f) \in I \widetilde{\Omega}_{\widehat{V_i}/K}$, after passing to completions (see \cite{Kunz_diff} for a definition of the module of differentials in the complete case). Since $\widehat{V_i}$ is a DVR containing a field, there exists a parameter $t_i \in \widehat{V_i}$ and a field $\ell_i$ such that $\widehat{V_i} \cong \ell_i \ps{t_i}$. Then
\[
\ds d_{\widehat{V_i}/K}(f) \in I \widetilde{\Omega}_{\widehat{V_i}/K} = I \widehat{V_i} dt_i,
\]
and thus $\partial(f)/\partial(t_i) \in I \cdot \widehat{V_i}$. This means that
\begin{equation} \label{equation_rees_valuations}
\ds v_i(f) = v_i(\partial(f)/\partial(t_i))+1 \gs v_i(I)+1,
\end{equation}
for all Rees valuations $v_i$ of $I$.
Now write $I=(f,g_1,\ldots,g_t)$, for some elements $g_i \in I$, and let $\mathcal{R} = S[It]_{(ft)} = S\left[\frac{g_1}{f},\ldots \frac{g_t}{f}\right]$ be the homogeneous localization of the Rees algebra at $(ft)$. We claim that $f$ is a unit in $\mathcal{R}$. If not, then $f \in \p$ for some $\p$ of height one. Let $\mathcal{S}$ be the integral closure of $\mathcal{R}$. Then $\mathcal{S}$ is finite over $\mathcal{R}$, since $\mathcal{R}$ is excellent. Then, there exists $Q \in \Spec(\mathcal{S})$ such that $Q \cap \mathcal{R} = \p$. Then $\mathcal{S}_Q$ is a DVR, with valuation $v$. In particular, $v$ is a Rees valuation of $I$, and $v(f) = v(I)$. This contradicts (\ref{equation_rees_valuations}). Therefore $1/f \in \mathcal{R}$, which means that there exists a polynomial $F(T_1,\ldots,T_m)$, say of degree $n+1$, such that $1/f = F(g_1/f,\ldots,g_m/f)$. It follows that
\[
\ds f^n = f^{n+1} \cdot \frac{1}{f} = f^{n+1} \cdot F\left(\frac{g_1}{f},\ldots,\frac{g_m}{f}\right) \in I^{n+1},
\]
since the numerator of $F\left(\frac{g_1}{f},\ldots,\frac{g_m}{f}\right)$ is a polynomial of degree $n+1$ in the $g_i$'s. Finally, assume (\ref{Hubl2,3}), so that $f+I^2$ is nilpotent in $G=\rm{gr}_I(S)$. This means that $f+I^2 \subseteq \mathcal{Q}$ for all $\mathcal{Q} \in \Spec(G)$. Let $\p$ be a minimal prime over $I$, then $G_\p = \rm{gr}_{\p S_\p}(S_\p)$, because $I S_\p = \p S_\p$. Since $S_\p$ is a regular local ring, with maximal ideal $\p S_\p$, we have that $G_\p$ is a domain. As a consequence, $\ker(G \to G_\p)$ is a prime ideal of $G$, so that $f+I^2 \in \ker(G \to G_\p)$. Since this happens for all $\p \in \Min(I)$, we have that $f + I^2 \in \cap_{\p \in \Min(I)} \ker(I/I^2 \to (I/I^2)_\p) =  I^{(2)}/I^2$.

\end{proof}

{\cb
\begin{corollary}
	Let $K$ be a field of characteristic zero, and let $S$ be a smooth algebra, essentially of finite type over $K$. Given a radical ideal $I \subseteq S$, there exists $N>0$ such that $\left( I^{(2)} \right)^{n} \subseteq I^{n+1}$ for all $n \geqslant N$.
\end{corollary}
}

Theorem \ref{Hubl2} leads to the following new conjecture, which is equivalent to the Eisenbud-Mazur Conjecture \ref{EisMaz_conj} under the assumptions of Theorem \ref{Hubl2}.

\begin{conjecture}[{H{\"u}bl \cite[Conjecture 1.3]{Hubl}}]
\label{conj_Hubl} 
Let $(R,\m)$ be a regular local ring, and let $I \subseteq R$ be a radical ideal. Let $f \in I$ be such that $f^n \in I^{n+1}$, for some $n \in \NN$. Then $f \in \m I$.
\end{conjecture} 
While Conjecture \ref{EisMaz_conj} is known to be false for rings of positive characteristic, we are not aware of any counterexamples to Conjecture \ref{conj_Hubl}.

\subsection{Intersection of symbolic powers and Serre's intersection multiplicity}

In this section, we discuss an intriguing connection between symbolic powers and local intersection theory. Let $(R,\m)$ be a regular local ring and $\p,\q \in \Spec R$ such that $\sqrt{\p+\q}=\m$. Most of the following materials follow from the work of Sather-Wagstaff  \cite{Wagstaff02}. Motivated by intersection multiplicities of two subvarieties inside an ambient affine or projective space, Serre proposed the following definition.

\begin{definition}[{Serre \cite{SerreLA}}]
The intersection multiplicity of $R/\p$ and $R/\q$ is defined by
$$
 \chi^R(R/\p,R/ \q) := \sum_{0}^{\dim R} \lambda (\Tor_i^R(R/\p,R/\q)).
$$
\end{definition}

The definition makes sense since all the Tor modules have finite length. For unramified regular local rings (for instance, if $R$ contains a field), Serre proved  the following properties.

\begin{theorem}[{Serre \cite{SerreLA}}]
Let $R$ be an unramified regular local ring and $\p,\q\in \Spec R$ such that $\sqrt{\p+\q}=\m$. Then we have:
\begin{enumerate}
\item (Dimension inequality) $\dim R/\p +\dim R/\q \ls \dim R$;
\item (Non-negativity) $\chi^R(R/\p,R/\q)\gs 0$;
\item (Vanishing) $\chi^R(R/\p,R/\q) =0$ if $\dim R/\p +\dim R/\q < \dim R$;
\item (Positivity) $\chi^R(R/\p,R/\q) >0$ if $\dim R/\p +\dim R/\q = \dim R$.
\end{enumerate}
\end{theorem}

These results and their potential extensions have been considered by many researchers  over the last fifty years. For a comprehensive account we refer to P. Roberts' book on this topic \cite{RobertsBook}. 

The non-negativity property part of Serre's Theorem was  extended to all regular local rings by Gabber \cite{Ber}. In an attempt to extend Gabber's argument to prove the positivity property, which still remains a conjecture, Kurano and Roberts proved
the following theorem.

\begin{theorem}[{Kurano-Roberts \cite{KuranoRoberts}}]
Let $(R,\m)$ be ramified regular local ring and $\p,\q\in \Spec R$ such that $\sqrt{\p+\q}=\m$. Assume that the positivity property holds for $R$ and $\dim R/\p +\dim R/\q = \dim R$. Then for each $n\gs 1$, $\p^{(n)}\cap \q \subseteq \m^{n+1}$. 
\end{theorem}

Thus, the simple containment $\p^{(n)}\cap \q \subseteq \m^{n+1}$ when $\dim R/\p +\dim R/\q = \dim R$ is a consequence of the positivity property. They conjectured that this holds for all regular local rings. Although this also stays open,  a stronger statement was proved for regular local rings containing a field by Sather-Wagstaff.

\begin{theorem}[{Sather-Wagstaff \cite[Theorem 1.6]{Wagstaff02}}]\label{SSW}
Let $(R,\m)$ be a regular local ring containing a field. Let $\p,\q\in \Spec R$ be such that $\sqrt{\p+\q}=\m$ and $\dim R/\p +\dim R/\q = \dim R$. Then 
for $m,n>0$, $\p^{(m)} \cap \q^{(n)}\subseteq \m^{m+n}$.
\end{theorem}

Note that without the condition $\dim R/\p +\dim R/\q = \dim R$, the conclusion fails easily. As an example, take $R=\CC[[x,y,z]]$, $\p =(x,y), \q=(y,z)$, then $y^n \in \p^{(n)} \cap \q^{(n)}$ for each $n$ but $y^n\notin \m^{2n}$. Also, taking $\p =(x,y), \q=(z)$ we see that the exponent $m+n$ is sharp. 

The proof of Theorem \ref{SSW} relies on the following interesting result, which itself can be viewed as a generalization of Serre's dimension inequality above. For a local ring $A$, let $e(A)$ denote the Hilbert-Samuel multiplicity with respect to the maximal ideal.

\begin{theorem}[{Sather-Wagstaff \cite[Theorem 1.7]{Wagstaff02}}]\label{SSW2}
Let $(A,\n)$ be a quasi-unmixed local ring containing a field and $P, Q\in \Spec A$ such that $A/P, A/Q$ are analytically unramified. Suppose that $\sqrt{P+Q} = \n$ and $\e(A) < \e(A_{P})+\e(A_{Q})$. Then $\dim A/P+\dim A/Q\ls \dim A$. 
\end{theorem} 

Given the above Theorem, \ref{SSW} follows readily. First we can pass to the completion of $R$ using some minimal primes lying over $\p,\q$. Thus we may assume $R$ is complete.  Let $f\in \p^{(n)} \cap \q^{(n)}$. Suppose that $f\notin \m^{m+n}$. Then set $A=R/fR$ and $P=\p A, Q=\q A$. It is clear that $\e(A_P)\gs m, \e(A_Q)\gs n$ and $\e(A)< m+n$. Thus by Theorem \ref{SSW2}, we have 
$$ \dim R/\p +\dim R/\q = \dim A/P + \dim A/Q\ls  \dim A <\dim R,$$ which gives
a contradiction.

%% file: UniformBounds.tex
\vspace{.4cm}

\section{Uniform Symbolic Topologies Property}
\label{SecUniform}

\subsection{Background on uniformity}

In this subsection we present a few results on uniformity in commutative algebra. We refer to \cite{HunekeRaicu} for a survey on this topic. 

In Section 2.2 we discussed an important uniformity result regarding symbolic powers of prime ideals in regular local rings (Theorems \ref{USP-Poly-CharZero} and \ref{USP-Poly-CharP}). The uniformity results we discuss in this section will be necessary to discuss results of the same flavor as Theorems \ref{USP-Poly-CharZero} and \ref{USP-Poly-CharP} for more general classes of rings.

We start by recalling some assumptions under which the Uniform Brian\c{c}on-Skoda Theorem and Uniform Artin-Rees Lemma hold. 

\begin{hypothesis}\label{Hyp}
We consider a Noetherian reduced ring  $R$ satisfying one of the following conditions
\begin{enumerate}
\item $R$ is essentially of finite type over an excellent ring containing in a field;
\item $R$ is $F$-finite;
\item $R$ is essentially of finite type over $\mathbb{Z}$;
\item $R$ is an excellent Noetherian ring which is the homomorphic image of a regular ring $R$ of finite Krull dimension such that for all $P$, $R/P$ has a resolution of singularities obtained by blowing up an ideal.
\end{enumerate}
\end{hypothesis}
The following results play an important role in the proof of many uniformity results about symbolic powers.

\begin{theorem}[{Uniform Brian\c{c}on-Skoda Theorem, Huneke  \cite{HunekeUniform}}] \label{ThmUBS} 
Let $R$ be a ring satisfying Hypothesis \ref{Hyp}.  Then, there exists $c = c(R)$ such that, for all ideals $I\subseteq R$ and all $n \gs 1$, $\overline{I^{n+c}} \subseteq I^n$.
\end{theorem}

\begin{theorem}[{Uniform Artin-Rees Lemma, Huneke \cite{HunekeUniform,HunekeUAR}}] \label{ThmUAR} 
Let $R$ be a ring satisfying Hypothesis \ref{Hyp}, and let $N$ and $M$ be $R$-modules.  Then, there exists $c$ such that 
\[
I^nM\cap N\subseteq I^{n-c}N
\]
for every ideal $I\subseteq R$, and $n\geq c$.
\end{theorem}

The Uniform Artin-Rees Lemma is not effective, since we cannot explicitly compute the integer $c$ in the previous theorem.

We now focus on a couple of results about associated primes, which are helpful while dealing with powers of ideals.

\begin{exercise}[{Yassemi \cite[Corollary 1.7]{Yassemi}}]\label{ExContractionAssPrimes}
Let $R\subseteq S$ be Noetherian rings, and $M$ be an $S$-module. Then, $\Ass_R M=\{\q \cap R\;|\; \q \in \Ass_S M\}.$ 
\end{exercise}

\begin{theorem}[{Brodmann \cite{BroAssympAssociatedPrimes}}]
\label{ThmBro}
Let $R$ be a Noetherian ring, and $I$ be an ideal. Then,
\[
A(I):=\bigcup_{n\in \NN} \Ass_R \left( R/ I^n \right)
\]
 is a finite set.
\end{theorem}
\begin{proof}
Let $\cR=\oplus_{n\in\NN} I^n $ denote the Rees algebra associated to $I$.
We know that $\cR$ is a Noetherian algebra. Then, $\cR/I \cR=\oplus_{n\in\NN} I^n/I^{n+1} $ is a finitely generated $\cR$-module,
and so, $\Ass_{\cR} \left(\cR/I\cR\right)$ is finite. As a consequence, we have that $\Ass_{R} \left(\cR/I\cR \right)=\bigcup \Ass_R   \left(I^n/I^{n+1} \right)$ is finite by Exercise  \ref{ExContractionAssPrimes}. From the short exact sequences
$0\to I^{n}/I^{n+1}\to R/I^{n+1}\to R/I^{n}\to 0$, we obtain that  
\[
\Ass_R  \left( R/I^{n+1} \right) \subseteq \Ass_R  \left( R/I^{n}\right) \cup \Ass_R  \left( I^n /I^{n+1} \right)
\] 
for every $n\in\NN.$
Then, 
$\bigcup_{n\in \NN} \Ass_R(R/ I^n)\subseteq \bigcup_{n\in\NN} \Ass_R \left( I^n/I^{n+1} \right)$, and the claim follows.
\end{proof}




\subsection{Linear Equivalence of Topologies}

It is helpful to introduce slightly more general topologies. This generality is useful because
under some standard operations such as completion, prime ideals do not necessarily stay
prime. Completion cannot ultimately be avoided as Theorem \ref{ThmSc2} demonstrates.

Let $A$ be a Noetherian ring and $I,J\subseteq A$ ideals. 
Write $I^n : \left<J\right> := \bigcup_{m\gs 1}(I^n : J^m)$. This ideal is the saturation of $I$
with respect to $J$. In terms of the primary decomposition of $I$, this saturation removes all
components which contain $J$. Since the number of associated primes of powers of $I$ are
finite in number by Theorem \ref{ThmBro}, it follows that symbolic powers are always saturations with respect to a fixed suitable
ideal.

Work byf Schenzel \cite{SchenzelTopEquiv,Sc2}, and Huckaba \cite{Huc} investigates conditions under which the $I$-adic and $\{I^n:\left<J \right>\}$ 
topologies are equivalent. Then, one requires for each $n\gs 1$ an integer $m\gs 1$ so that $I^m:\left<J\right> \subseteq I^n$. In particular, a theorem of Schenzel \cite[Theorem 3.2]{Sc2}
says when certain ideal topologies are equivalent. 
In this section, we write $A(I)$ for the union over $n$ of the associated primes of $I^n$, 
 a finite set of prime ideals. We also write $\widehat{T}$ for the completion of the local 
ring $T$ with respect to its maximal ideal. Here is Schenzel's Theorem: 

\begin{theorem}[{Schenzel \cite[Theorem 3.2]{Sc2}}]\label{ThmSc2}
Let $A$ be a Noetherian ring and $I, J\subseteq A$ two ideals. 
Then the following are equivalent.
\begin{enumerate}
\item The $\{I^n:\left<J\right>\}$ topology is equivalent to the $I$-adic topology.
\item  $\dim(\widehat{R_\p}/(I\widehat{R_\p}+z)) > 0$, for all prime ideals $\p\in A(I)\cap V(J)$, and prime ideals $z\in \Ass(\widehat{R_\p})$.
\end{enumerate}
\end{theorem}

We say that the topology determined by $\{I^n:\left<J\right>\}$ is \it linearly equivalent \rm to the
topology determined by $I^n$ if there is a constant $h$ such that for all $n\gs 1$,
$$I^{hn}:\left<J\right>\subseteq I^n.$$
This concept is  \it a priori \rm stronger than having equivalent topologies. However,  
Swanson  \cite{Swanson} proved the following beautiful result relating the notions of equivalent and linearly equivalent 
topologies.

\begin{theorem}[{Swanson \cite[Main Theorem 3.3.]{Swanson}}]\label{IS} 
Let $A$ be a Noetherian ring and $I,J$ ideals. Then the $\{I^n:\left<J\right>\}$ 
and $I$-adic topologies are equivalent if and only if there exists $h\gs 1$ such that, for all $n\gs 1$, 
$I^{hn} : \left<J\right> \subseteq I^n$. 
\end{theorem}

Two comments are in order regarding Theorem \ref{IS}. The first is that $h$ depends \emph{a priori} on $I$.
The second is that the theorem implies that if $\p \subseteq A$ is a prime ideal, and the 
$\p$-symbolic and $\p$-adic topologies are equivalent, then there exists $h\gs 1$ so that 
$\p^{(hn)}\subseteq \p^n$, for all $n$. Swanson's Theorem sets the stage for uniform bounds for regular rings \cite{ELS,HHpowers} discussed in Subsection \ref{SubsecRegRingUniformBounds}. In particular,
if $d$ is the dimension of the regular local ring $R$,
then $\p^{(dn)}\subseteq \p^n$ for all $n\gs 1$ and all primes $\p$. The point is that while $h$ in Theorem \ref{IS} 
may depend on 
the ideal $I$, $h = d$ in Theorem \ref{USP-Poly-CharZero} and \ref{USP-Poly-CharP}  
is independent of the ideal. 

One could reasonably ask to do even better: could there be an integer $t$ such that
for every prime $\p^{(n+t)}\subseteq \p^n$? Even very simple examples
show that this is hopeless. For example, in the hypersurface $x^2-yz = 0$, the prime ideal $\p = (x,y)$
has $2n^{\text{th}}$ symbolic power generated by $y^n$, which is  not in $\p^{2n-t}$ for any fixed constant $t$, as
 $n$ gets large. Even the results of Theorem \ref{USP-Poly-CharZero} and \ref{USP-Poly-CharP} 
 cannot be improved in an asymptotic sense,  as shown
by Bocci and Harbourne \cite{BoH}.  All of these results lead to the following question:

\begin{question} \label{question} Let $(R,\m,K)$ be a complete local domain. Does there exist a positive integer $b=b(R)$ such that $\p^{(bn)}\subseteq \p^n$, for all prime ideals $\p\subseteq R$ and all $n\gs 1$?
\end{question}

We say that a ring satisfies the Uniform Symbolic Topologies Property, abbreviated as USTP, if the answer to the previous question is affirmative.

We now show that complete local domains have the property that the adic and symbolic topologies are linearly equivalent for all prime ideals. However, the constant $c$ which gives the relation $\p^{(cn)}\subseteq \p^n$ depends {\it a priori} on the ideal $\p$. We first make an observation needed for this result.

\begin{remark}\label{exponents} 
Let $R$ be complete local domain.
For ideals $I,J \subseteq R$, suppose there exist integers $d,t \gs 1$ such that $I^{dn}:\langle J\rangle \subseteq I^{n-t}$, for all $n \geq t$. An induction argument shows that for $c := d(t+1)$, $I^{cn} : \langle J\rangle \subseteq I^n$ for all $n\gs 1$. 
\end{remark}

We now show that Question \ref{question} is well-posed.

\begin{proposition}[{Huneke-Katz-Validashti \cite[Proposition 2.4]{HKV}}] \label{equivalent} 
Let $(R,\m,K)$ be a complete local domain and let $\p\subseteq R$ be a prime ideal. Then the $\{\p^{(n)}\}$ topology is linearly equivalent to the $\{\p^n\}$ topology. In particular, there exists $c > 0$ (depending on $\p$) such that $\p^{(cn)}\subseteq \p^n$, for all $n\gs 1$. 
\end{proposition}

\begin{proof} 
It suffices to prove the second statement. Let $S$ denote the integral closure of $R$ and set 
$I := \sqrt{\p S}$.
Since $S$ is an excellent normal domain, it is locally analytically normal, so the completion of $S_\q$ is a domain for
all primes $\q$. In particular, by Theorem \ref{ThmSc2}, the 
$\{I^{(n)}\}$ topology is equivalent to the $\{I^n\}$ topology. Here we are writing $I^{(n)}$ for $I^n_U \cap S$, where $U := S\backslash \q_1\cup \cdots \cup \q_r$, for 
$\q_1,\ldots, \q_r$ the primes in $S$ lying over $\p$, so that $I = \q_1\cap \cdots \cap \q_r$. Thus, by \ref{IS}, 
$\p^{(kn)}\subseteq I^{(kn)}\subseteq I^n$ for some fixed $k$ and all $n\gs 1$. On the other hand, there is an $e$ such that $I^e\subseteq \p S$,
so that \[\p^{(ken)}\subseteq \p^nS\cap R\subseteq \p^{n-l},\] 
for some $l$ by Artin-Rees. By Remark \ref{exponents}, taking $c := ke(l+1)$ gives $\p^{(cn)}\subseteq \p^n$, for all $n$. 
\end{proof}

It is worth noting that the previous result is motivated by the work of Swanson.

\begin{theorem}[{Swanson \cite[Theorem 3.4]{Sw1}}] 
\label{ThmISPrimary}
Let $R$ be a Noetherian ring and $I\subseteq R$ be an ideal.
Then, there exists a positive 
integer $c$ and for each $n$ an irredundant primary decomposition 
$\q_{n,1}\cap \cdots \cap \q_{n, s_n}$ of $I^n$ so that for $\p_j := \sqrt{(\q_{n,j})}$, 
$\p_j^{cn}\subseteq \q_{n,j}$, for all $n\gs 1$ and $1\ls j\ls s_n$.
\end{theorem}

\subsection{A Uniform Chevalley Theorem}
In order to study the Uniform Symbolic Property, it is  useful to have a general version of the Zariski-Nagata Theorem (Theorem \ref{ThmGralZN}). If $(R,\m,K)$ is not regular, one cannot guarantee that $\p^{(n)}\subseteq \m^n$. Our goal is to show that for complete local domains there exists a constant $h$, not depending of $\p$, such that $\p^{(hn)}\subseteq \m^n,$ which can be seen as a uniform version of 
Chevalley's theorem.

\begin{theorem}[{Chevalley}]
Let $(R,\m,K)$ be a complete local ring, $M$ be a finitely generated module,
and let $\{ M_i \}$ be a nonincreasing sequence of submodules. 
Under these conditions, $\displaystyle\bigcap_{i\in\NN} M_i=0$ if and only if
for every integer $t$ there exists $i$ such that $M_i\subseteq \m^t M.$
\end{theorem}

As a consequence, if $\{J_n\}_{n\gs 1}$ is a descending collection of ideals with $\bigcap_{n\gs 1} J_n = 0$, then 
the $\{J_n\}$ topology is finer than the $\m$-adic topology. In other words, 
for all $n\gs 1$, there exists $t\gs 1$, such that $J_t\subseteq \m^n$.  Once again, we would like
to understand when the symbolic topology is finer than the $\m$-adic topology in not necessarily
complete rings. To do so, it is convenient to introduce another type of saturation.

\begin{notation} 
Let $S\subseteq R$ be a multiplicatively closed set and $L \subseteq R$ an ideal. We write 
$L:\left< S \right>$ for $LR_S\cap R$. 
\end{notation}

Let $I, J$ be ideals of $R$ and take $s \in J$ with the following property: for all $\p \in A(I)$, 
$s\in \p$ if and only if $J\subseteq \p$. Then $I^n : \left<S\right> = I^n : \left<J\right>$, 
for all $n\gs 1$. Consequently, any result about the $\{I^n:\left<S\right>\}$ 
topology recovers the corresponding result about the $\{I^n:\left<J\right>\}$ topology. Moreover, if 
we let $S$ denote the complement of the union of the associated primes of $I$, then by definition, 
$I^{(n)} = I^n : \left<S\right>$, for all $n$. 

The following proposition is implicit in the work of McAdam \cite{Mc2} and Schenzel \cite{Sc2}. We present the proof given by Huneke, Katz and Validashti \cite{HKV}.

\begin{proposition}[{Huneke-Katz-Validashti \cite[Proposition 2.2]{HKV}}]\label{finerthanm} 
Let $(R,\m,K)$ be a local ring with completion $\widehat{R}$. Let $I\subseteq R$ be an 
ideal and $S\subseteq R$ a multiplicatively closed set. Write $\q_1,\ldots, \q_s$ for 
the associated primes of $\widehat{R}$. Then the $\{I^n : \left< S \right>\}$ topology 
is finer than the $\m$-adic topology if and only if $(I\widehat{R}+\q_i)\cap S = \emptyset$  for all $1\ls i\ls s$.
\end{proposition}

\begin{proof} 
Since $\widehat{R}$ is faithfully flat over $R$, the $\{I^n :_R \left< S \right>\}$ topology is finer than the $\m$-adic topology if and only if 
the $\{I^n\widehat{R} :_{\widehat{R}} \left< S \right>\}$ topology is finer than the $\m\widehat{R}$-adic topology. Thus, we may assume 
that $R$ is a complete local ring with associated primes $\q_1, \ldots, \q_s$. By Chevalley's Theorem, 
the $\{I^n : \left< S \right>\}$ topology is \emph{not} finer than the $\m$-adic topology if and only if \[\bigcap_{n\gs 1} (I^n : \left< S \right>) 
\not = 0.\]
Suppose  $0 \not = f \in \bigcap_{n\gs 1} (I^n : \left< S \right>).$ Then for each $n\gs 1$, there 
exists $s \in S$ such 
that $s \in (I^n:f)$. By applying the Artin-Rees Lemma to $I^n\cap (f)$, we see that for $n$ large, $s \in (0:f) + I^{n-k}$, for some $k$. 
Taking $\q_j$ so that $(0:f) \subseteq \q_j$, we have 
that $(I+\q_j)\cap S \not = \emptyset$. 

On the other hand, suppose that $(I+\q_j)\cap S \not = \emptyset$, for some $j$. Then for all 
$n\gs 1$, $(I^n+\q_j)\cap S \not = \emptyset$. Let $\q_j = (0:f)$. 
Then for each $n$, there exists $s \in S$ such that $s \in I^n + (0:f)$, i.e., $sf\in I^n$. Thus, 
$0 \not = f \in \bigcap_{n\gs 1} (I^n : \left< S \right>).$ 
\end{proof}

Since $R$ is analytically irreducible,  $I = \p$ is prime. By taking  $S = R - \p$ in the previous  proposition, we obtain that the $\{\p^{(n)}\}$ topology is finer that the $\m$-adic topology.

We can now state and show the Uniform Chevalley Theorem.

\begin{theorem}[{Huneke-Katz-Validashti  \cite[Theorem 2.3]{HKV}}]\label{linearch} 
Let $R$ be an analytically unramified local ring. Then there exists $h \gs 1$ 
with the following property: for all ideals $I\subseteq R$ and all multiplicatively closed sets $S\subseteq R$ such that 
the $\{I^n:\left< S \right>\}$ topology is finer than the $\m$-adic topology, 
$I^{hn}:\left< S \right>\subseteq \m ^n$, for all $n\gs 1$. 
\end{theorem} 
\begin{proof}[Proof Ideas:]
The statement reduces to the complete local domain case. In this case, we applied Rees' theory of
Rees valuations and degree functions \cite[Theorem 2.3]{R} together with a theorem of Izumi \cite{Iz} to obtain
the result.  
\end{proof}

As a corollary, we  globalize the statement in Theorem \ref{linearch}. 

\begin{corollary}[{Huneke-Katz-Validashti  \cite{HKV}}]\label{global} 
Let $R$ be a Noetherian ring and $J\subseteq R$ an ideal. 
Suppose that $R_\p$ is analytically unramified for all $\p\in A(J)$. Then there exists 
a positive integer $h$ with the following property: for all ideals 
$I\subseteq R$ and multiplicatively closed sets $S$ for which the 
$\{I^n : \left< S \right>\}$ topology is finer than the $J$-adic topology, 
$I^{hn} : \left< S \right> \subseteq J^n$, for all $n$.
\end{corollary}

\begin{proof} 
The point of the proof is that we can combine Theorem \ref{linearch} 
with Theorem \ref{ThmISPrimary}. It follows 
from this that $\p_j^{(cn)}\subseteq \q_{n,j}$, for all $n$ and all $j$. 

On the other hand, if $I$, $S$ are such that the $\{I^n : \left< S \right>\}$ topology 
is finer than the $J$-adic topology, then the $\{I^n : \left< S \right>\}$ topology 
is finer than the symbolic topology $\{\p^{(n)}\}$ for any $\p\in A(J)$. 
By our hypothesis on $A(J)$ and Theorem \ref{linearch}, there is a positive 
integer $l$ such that $(I^{ln}:\left<S\right>)_\p \subseteq \p^n_\p$, for 
all $n$ and all $\p\in A(J)$. Therefore, $I^{ln} : \left< S \right> \subseteq \p^{(n)}$, 
for all $\p\in A(J)$ and $n\gs 1$. Combining this with the previous paragraph and 
setting $h := cl$, it follows 
that $I^{hn}:\left<S\right>\subseteq J^n$, for all $I,S$ and $n\gs 1$.
\end{proof}

\subsection{Isolated Singularities}
The purpose of this subsection is to discuss a theorem of
Huneke, Katz and Valisdashti \cite{HKV} which proves that for a large class 
of isolated singularities, the symbolic topology defined by a prime ideal $\p$ is 
uniformly linearly equivalent to the $\p$-adic topology. We sketch their proof that for such isolated singularities $R$, there exists 
$h\gs 1$, independent of $\p$, such that for all primes $\p\subseteq R$, $\p^{(hn)}\subseteq \p^n$, 
for all $n$.  The following theorem is our focus:


\begin{theorem}\label{mainintro}[Huneke-Katz-Validashti  \cite{HKV}] Let $R$ be an equicharacteristic local domain 
such that $R$ is an isolated singularity. Assume that $R$ is either essentially of finite 
type over a field of characteristic zero or $R$ has positive characteristic and is $F$-finite. 
Then there exists $h\gs 1$ with the following property. For 
all ideals $I\subseteq R$ such that the symbolic topology of $I$ is equivalent to the 
$I$-adic topology, $I^{(hn)}\subseteq I^n$, for all $n\gs 1$.
\end{theorem}

There are three crucial ingredients in the proof of this theorem:
the relation between the Jacobian ideal and symbolic powers established in \cite{HHpowers}, the uniform Artin-Rees Theorem (Theorem \ref{ThmUAR}), and a uniform Chevalley Theorem (Theorem \ref{linearch}). 

We now focus on a result needed to prove uniform bounds for isolated singularities. This result comes from
the work of Hochster and Huneke \cite{HHpowers}, where the following property of the Jacobian ideal plays a crucial role 
in the main results concerning uniform linearity of symbolic powers over regular local rings:

\begin{theorem}[{Hochster-Huneke \cite[Theorem 4.4]{HHpowers}}]\label{HH} Let $R$ be a an equidimensional local ring essentially of finite type over a field $K$ of 
characteristic zero. Let $J$ denote the square of the Jacobian ideal of $R$ over $K$. Then there exists $k\gs 1$ 
such that \[J^n I^{(kn+ln)}\subseteq (I^{(l+1)})^n,\] for all ideals $I$ with 
positive grade and all $l, n \gs 1$.
\end{theorem}

Moreover, a parallel result is proved for $F$-finite local rings in characteristic $p$
with isolated singularity \cite{HKV}. In this case, $J$ can be chosen to be a fixed $\m$-primary ideal, though
not necessarily the square of the Jacobian ideal. 

The main point of the proof of the main result on isolated singularities is the following theorem. 

\begin{theorem}[Hunke-Katz-Validashti \cite{HKV}]
\label{lem1} 
Let $R$ be a Noetherian ring in which the 
uniform Artin-Rees lemma holds and $J \subseteq R$ an
ideal with positive grade. Assume that $R_\p$ is analytically unramified for all $\p\in A(J)$. 
Let $\mathcal{J}$ denote the collection of ideals $I \subseteq R$ for which the
$\{I^{(n)}\}_{n\gs 1}$ topology is finer than the $J$-adic topology. Suppose there exists $k \gs 1$ 
with the following property: for all ideals
$I \in \mathcal{J}$,
\[J^n I^{(kn+ln)}\subseteq (I^{(l+1)})^n,\] for all $l, n \gs 1$. Then there exists a positive integer $h$ such 
that for all ideals $I \in \mathcal{J}$, $I^{(hn)}\subseteq I^n$, for all $n$.
\end{theorem}
\begin{proof}[Sketch of proof:]
By Corollary \ref{global},
we can choose $h_0>k$ so that $I^{(h_0n)}\subseteq J^n$, for all ideals $I \in \mathcal{J}$.
Taking $l = 0$ in Theorem \ref{HH} 
gives $J^n I^{(kn)} \subseteq I^n$, for all $n$ and all ideals $I \in \mathcal{J}$. Thus, 
$
I^{(h_0n)} I^{(kn)}\subseteq I^n
$ 
for all  $n$ and all ideals $I \in \mathcal{J}$. 
On the other hand, if $n = 2$ and $l = h_0r$ in Theorem \ref{HH},
we get 
$$
J^2 I^{(2k+2h_0r)}\subseteq (I^{(h_0r+1)})^2\subseteq I^{(h_0r)} I^{(kr)} 
\subseteq I^r,
$$
which holds for all $r$. Thus there exists a positive integer $B$ such that $J^2 I^{(Br)} \subseteq I^r$ for all $r$.
Choose a non-zerodivisor $ c \in J^2$.
Then $c I^{(Bn)}\subseteq I^n$, for all $n$. 
Thus, 
$I^{(Bn)}\subseteq (I^n:c)$ for all $n$ and all $I$. 
By the uniform 
Artin-Rees Theorem (see Theorem \ref{ThmUAR}),
we find
$q \gs 1$ with the property that 
\[I^{(Bn+Bq)}\subseteq (I^{n+q}:c)\subseteq I^n,\]
for all $n$ and all ideals $I \in \mathcal{J}$. Taking $h := B+Bq$ gives 
$I^{(hn)}\subseteq I^n$, for all $I \in \J$ and all $n\gs 1$, as required.
\end{proof}

We now can show the main result in this subsection.

\begin{theorem}[{Huneke-Katz-Validashti \cite[Theorem 1.2]{HKV}}]
\label{main} 
Let $R$ be an equicharacteristic reduced local ring such that $R$ is an isolated singularity. 
Assume either that $R$ is equidimensional and essentially of finite type over a field of characteristic zero, or that 
$R$ has positive characteristic and is $F$-finite. 
Then there exists $h\gs 1$ with the following property: for all ideals $I$ with positive grade for which the
$I$-symbolic and $I$-adic topologies are equivalent,
$I^{(hn)}\subseteq I^n$, for all $n\gs 1$.
\end{theorem}

\begin{proof} The ring $R$ is excellent in both cases, and so $R$ is analytically unramified (see \cite{F-finExc} for the $F$-finite case). Let $d=\dim(R).$

Suppose first that $R$ is essentially of finite type over a field of 
characteristic zero.
 Let $J$ denote the square of the Jacobian ideal. By Theorem \ref{HH}, 
$J^n I^{(dn+ln)}\subseteq (I^{(l+1)})^n$
for all ideals $I$  with positive grade and all $n\gs 1$ and all $l \gs 0$. Thus, since $J$ is $\m$-primary and 
$R$ is analytically unramified, we may use Theorem \ref{lem1} with $k = d$ to 
obtain the desired result. 

The positive characteristic case follows similarly once  an ideal
$J$ is constructed to play a similar role to the Jacobian ideal. \end{proof}

\medskip

As a consequence of the previous theorem, we obtain the Uniform Symbolic Topologies  Property for isolated singularities.

\begin{theorem}[[Hunke-Katz-Validashti \cite{HKV}]\label{HKV}
\label{prime} Let $R$ be an equicharacteristic local domain 
such that $R$ is an isolated singularity. Assume that $R$ is either essentially of finite 
type over a field of characteristic zero or $R$ has positive characteristic, is $F$-finite and 
analytically irreducible. 
Then there exists $h\gs 1$ with the following property: for 
all prime ideals $\p\not = \m$, $\p^{(hn)} \subseteq \p^n$, for all $n$. 
\end{theorem}

\subsection{Finite Extensions}

We now present a uniform relation between the extension of an ideal and its radical in the case of a finite extension. This is a key ingredient in the proof that the USTP descends for finite extensions.

\begin{lemma}[{Hunke-Katz-Validashti  \cite{HKVFinExt}}]\label{lemma for descent}
Let $R \subseteq S$ be a finite extension of domains, with $R$ integrally closed. Let $e=[S:R]$, i.e. the degree of the quotient field of $S$ over the quotient field of $R$. If $\q \in \Spec(R)$, then $\left(\sqrt{\q S}\right)^e \subseteq \overline{\q S}$. Furthermore,  if $e!$ is invertible, then $\left(\sqrt{\q S}\right)^e \subseteq \q S$.
\end{lemma}
\begin{proof}
Let $x \in \sqrt{\q S}$. There is a polynomial $f(T) = T^e+r_1T^{e-1} + \ldots + r_e$ such that $f(x) = 0$, and with $r_i \in \q$ for all $i$  \cite[Lemma 5.14 and Proposition 5.15]{Atiyah}. As a consequence,  $x^e \in \q S$. If $e!$ is invertible, we have that $\left(\sqrt{\q S} \right)^e = (x^e \mid x \in \sqrt{\q S}) \subseteq \q S$. 

We now consider the case where $e!$ is not invertible.  We have that elements of the form $x_1\cdots x_e$, for $x_i \in \sqrt{\q S}$, generate the ideal $\left(\sqrt{\q S}\right)^e$. Then, $(x_1\cdots x_e)^e = x_1^e \cdots x_e^e \in \left(\q S \right)^e$. This implies that $x_1\cdots x_e \in \overline{\q S}$, and thus, $\left(\sqrt{\q S}\right)^e \subseteq \overline{\q S}$.
\end{proof}

\begin{remark}\label{UBS remark}
If $R$ satisfies Hypothesis \ref{Hyp} and $\q^{(bn)} \subseteq \overline{\q^n}$, then there exists $t$ such that 
$\q^{b(t+1)n} \subseteq \q^n$ for all $n \gs 1$. This is a consequence of Theorem \ref{ThmUBS}.
\end{remark}

The following example shows that if the assumption that the extension is finite is dropped from Lemma \ref{lemma for descent}, there does not necessarily exist a uniform $c$ such that for all primes $\q$ in $R$, $\left( \sqrt{ \q S} \right)^c \subseteq \q S$. Computations using Macaulay2 \cite{M2} were crucial to find this example.

\begin{example}\label{uniform powers counterexample}
Let $R=K[a,b,c,d]/(ad-bc)$, which includes in $S=K[x,y,u,v]$ via the map $h: R \longrightarrow S$ given by $h(a) = xy, h(b) = xu, h(c) = yv, h(d) = uv$.

For each integer $A$, let $\q_A$ be the prime ideal in $R$ given by the kernel of the map $f_A\!: R \longrightarrow k[t]$, where $f_A$ is given by
$$f_A(a) = t^{4A}, f_A(b) = t^{4A+1}, f_A(c) = t^{8A+1}, f_A(d) = t^{8A+2}.$$
Let $Q_A = h \left( q_A \right) S$, and $g_A = x u^{4A+1} - v y^{4A+1}$. Then $\left( g_A \right)^{4A} \in Q_A$, but $\left( g_A \right)^{4A-1} \notin Q_A$. As a consequence, $\left( \sqrt{ Q_A} \right)^{4A-1} \nsubseteq Q_A S$.

To check this, fix $A$ and write $\q := \q_A$, $Q := Q_A$ and $g := g_A$. We note that

\begin{enumerate}[(1)]
\item $\q = \left(c-ab, d-b^2, b^{4A}-a^{4A+1}\right)$, so 
\item $Q = \left( y(v-x^2u), u(v-x^2u), x^{4A} \left(x y^{4A+1}-u^{4A} \right) \right)$.
\end{enumerate}

The fact that $g^{4A} \in Q$, but $g^{4A-1} \notin Q$ follows because $\lbrace y(v-x^2u), u(v-x^2u), x^{4A} \left( u^{4A} - y^{4A+1} \right) \rbrace$ is a Gr\"{o}bner basis for $Q$ with respect to the lexicographical order induced by the following order on the variables: $v > u > x > y$. This can be checked applying the Buchberger's algorithm on the generating set 
$$\lbrace f_1 = yv-yux^2, f_2 = u(v-x^2u) = vu-u^2x^2, f_3 = u^{4A} x^{4A} - y^{4A+1}x^{4A+1} \rbrace.$$

In this order, we now have that $\ini(Q) = \left( yv, uv, x^{4A}u^{4A} \right)$.
Moreover,
$$g^n =  \left( x u^{4A} - x^2 y^{4A+1} \right)^n = \sum_{i=0}^n c_i x^i u^{4Ai}x^{2(n-i)} y^{(4A+1)(n-i)}, \textrm{ where } c_i \in \mathbb{Z},$$
so that $g^n$ has leading term $x^n u^{4An}$. If $g^n \in Q$, then 
$$\ini \left( g^n \right) = x^{n} u^{(4A+1)n} \in \left( yv, uv, x^{4A}u^{4A} \right).$$
This happens if and only if $n \gs 4A$, meaning that $g^n \notin Q$ for all $n < 4A$. On the other hand,
$$g^{4A} = \left( x \left(u^{4A} - x y^{4A+1} \right) \right)^{4A} = x^{4A} \left(u^{4A} - x y^{4A+1} \right) \left(u^{4A} - x y^{4A+1} \right)^{4A-1} \in Q .$$
\end{example}

We now present the main result in this subsection.

\begin{theorem}[{Hunke-Katz-Validashti  \cite[
Corollary 3.4.]{HKVFinExt}}]\label{descent} 
\label{ThmFiniteExtensionHKV}
Let $R \subseteq S$ be a finite extension of domains, with $R$ integrally closed, such that both rings satisfy Hypothesis \ref{Hyp}. If $S$ has USTP, then $R$ has USTP.
\end{theorem}
\begin{proof}
We first show that there exists $r$ such that for all $\p \in \Spec(S)$ the following is true: if $\p^{(bn)} \subseteq \p^n$ for all $n$ and  $\q = \p \cap R$, then $\q^{(rbn)} \subseteq \q^n$ for all $n$. Note that $\q^{(n)} \subseteq \p^{(n)}$. It suffices to show that there exists $r$, independent of $\p$, such that $\p^{rn} \cap R \subseteq \q^n$. Indeed, this gives
\[
\q^{(rbn)} \subseteq \p^{(rbn)} \cap R \subseteq \p^{rn} \cap R \subseteq \q^n.
\]
By replacing $S$ by $\overline{S}$, it is enough to show our claim for $S$ integrally closed. It also suffices to show this separately for $R \subseteq T$ and $T \subseteq S$ separately, where $T$ is the integral closure of $R$ in some intermediate field $E$ with $K \subseteq E \subseteq L$, where $K$ is the fraction field of $R$ and $L$ is the fraction field of $S$.

We have two cases to consider:
\begin{enumerate}[(a)]
\item $L$ is purely inseparable over $K$.
\item $L$ is separable over $K$.
\end{enumerate}

Write $e = [S:R]$.

\begin{enumerate}[(a)]
\item $L$ is purely inseparable over $K$.

For every element $u \in S$, $u^k \in R$ for some $k$, and thus $Q = \sqrt{\q S}$.

If $u \in Q^{en} \cap R = \left( Q^e \right)^n \cap R$, then $u \in \left( \overline{\q S} \right)^n \cap R \subseteq \overline{\q^n}$.
We have that  $\overline{\q^n} \subseteq \q^{n-t}$ for $t$ by Theorem \ref{ThmUBS}. The claim now follows from Remark \ref{UBS remark}.

\item By perhaps extending $L$, we can assume $L$ is Galois over $K$. Write $\sqrt{\q S} = \p_1 \cap \ldots \cap \p_k$, and notice that $k \ls e$, since the $P_i$ are permuted by the Galois group.

Let $u \in Q^{e^2 n} \cap R$. Then
\[
u^e \in \left( u^k \right) \subseteq \p_1^{en} \ldots \p_k^{en} = \left( \p_1 \ldots \p_k \right)^{en} \subseteq \left( \sqrt{\q S} \right)^{en},
\]
and by Lemma \ref{lemma for descent}, $\left( \sqrt{\q S} \right)^{en} \subseteq \left( \overline{\q S} \right)^n \cap R$. Then,
\[
u^e \in \left( \overline{\q S} \right)^n \cap R \subseteq \overline{\q^n} \subseteq \q^{n-t}.
\]
\end{enumerate}
\end{proof}

\subsection{Direct Summands of Polynomial rings}
In this subsection we discuss uniform bounds for direct summands of polynomial rings. This includes  affine toric normal rings \cite{MelMonomials}, and  rings on invariants.   Rings corresponding
to the cones of Grassmannian,  Veronese and Segre  varieties are also direct summands of a polynomial ring. If one assumes that the field has characteristic
zero, then the ring associated to the $t \times t$ minors of an $n \times n$ generic matrix is also one of these rings. We continue with the strategy used to prove the Zariski-Nagata Theorem in Subsection \ref{SecZariskiNagata}.
We start by recalling a property of direct summands that was used by \`{A}lvarez-Montaner, the fourth, and fifth author in their study of $D$-modules.

\begin{lemma}[{\cite{AMHNB}}]\label{LemmaRestriction}
Let $R\subseteq S$ let be two finitely generated $K$-algebras.
Let  $\beta\!: S\to R$ be any $R$-linear morphism. Then, for every $\delta\in D^n_K(S)$, we have that $\tilde{\delta}:=\beta\circ\delta_{|_{R}}\in D^n_K(R).$
\end{lemma}

We now give a specific bound for the Chevalley Theorem for homogeneous ideals in a graded direct summand of a polynomial ring.

\begin{theorem}\label{ThmZNDS}
Let $K$ be a field, $S=K[x_1,\ldots,x_n]$,  $\eta=(x_1,\ldots,x_n)S$, $f_1,\ldots,f_\ell\in S$ be homogeneous polynomials, $R=K[f_1,\ldots,f_\ell]$,  $\m=(f_1,\ldots,f_\ell)R$, and $B=\max\{\deg(f_1),\ldots,\deg(f_\ell)\}$. 
Suppose that the inclusion, $R\subseteq S$, splits. Then,
$$
\q^{(Bn)} \subseteq \m^n
$$
for every homogeneous prime ideal $\q\subseteq R.$
\end{theorem}
\begin{proof}
We first show that $\q^{(Bn)}\subseteq \eta^{Bn}\cap R.$
Since $\q$ and $\eta^{Bn}\cap R$ are homogeneous ideals, it suffices to show that if a homogeneous element $f\not\in \eta^{Bn}\cap R,$
then $f\not\in \q^{(Bn)}.$

Since $\eta^{Bn}=\eta^\dsp{Bn}$ and $f$ is homogeneous, there exists an operator $\delta \in D^{n-1}_K(S)$ such that $\delta(f)=1.$ 
Then, 
$\beta\circ \delta (f)=1$ by Lemma \ref{LemmaRestriction}.
Since  
$\beta\circ \delta_{|_{R}} \in D^{n-1}_K(R)$, we have that $f\not\in \q^\dsp{Bn}.$  In particular, $f\not\in \q^{(Bn)}$, since $\q^{(Bn)}\subseteq \q^\dsp{Bn}$.

We now show that $\eta^{Bn}\cap R\subseteq \m^n.$
Let $g\in \eta^{Bn}\cap R.$
Then, $g$ is a linear combination of products $f^{\alpha_1}_1\cdots f^{\alpha_\ell}_\ell$ such that
$$
Bn\ls \deg(f)=\alpha_1\deg(f_1)+\ldots+\alpha_\ell\deg(f_\ell)\ls D(\alpha_1+\ldots + \alpha_\ell).
$$
Then, $\alpha_1+\ldots+\alpha_\ell\gs n.$
Hence, $g\in \m^n.$
We conclude that
$$
\q^{(Bn)}\subseteq  \eta^{Bn}\cap R\subseteq \m^n.
$$
\end{proof}

As a corollary of the previous result, we find that $2$ is a sufficient bound for determinantal rings.

\begin{corollary}\label{CorZNDet}
Let $X$ be a $n\times m$ generic matrix of variables, $K$ a field, $R=K[X]/I_t(X),$ and $\m=(x_{i,j})R.$ 
If either $t=2$ or $\Char(K)=0,$ then
$$
\p^{(2n)}\subseteq \m^n
$$
for every homogeneous prime ideal $\p\subseteq \m.$
\end{corollary}
\begin{proof}
If either $t=2$ or $\Char(K)=0,$ then $R$ is a direct summand of a polynomial ring, and it is generated by homogeneous polynomials of degree $2.$ The rest is a consequence of Theorem \ref{ThmZNDS}.
\end{proof}

We now focus on Question \ref{question} for direct summands of polynomial rings.
In case of direct summands whose extension is finite, we have specific values for the uniform bounds given in Theorem \ref{ThmFiniteExtensionHKV}. 

\begin{theorem}\label{UniformDirectSummand}
Let $S=K[x_1,\ldots,x_n]$ and $R\subseteq S$ a direct summand.
Suppose that $S$ is a finitely generated $R$-module.
Let $\p\subseteq R$ a prime ideal and $h=\Ht(\p).$
If $k=[S:R],$ then
$$
\p^{(khn)}\subseteq \p^{n-d}
$$
for every positive integer $n$.
Furthermore, if  $k!$ is invertible in $R$, then
$$
\p^{(khn)}\subseteq \p^{n}.
$$
\end{theorem}
\begin{proof}
Let $\q_1,\ldots,\q_\ell$ be the minimal primes of $\p S$.
Since $R\subseteq S$ is an integral extension, and $R$ is integrally closed, the going-up and going-down theorems apply. Then,
 $\q_i\cap R=\p.$
Let $V=S\setminus (\q_1\cup\ldots\cup \q_\ell)$.
Then, 
\begin{align*}
r\in R\setminus \p & \Longrightarrow r\not\in \q_i\cap R\\
& \Longrightarrow r\not\in \q_i \hbox{ for all }i  \\\
& \Longrightarrow r\in S\setminus (\q_1\cup\ldots\cup \q_\ell).\\
\end{align*}
Thus, $R\setminus \p\subseteq V.$

We have that
\begin{align*}
\p^{(khn)} &=\p^{knh} R_\p\cap R \subseteq \p^{khn}S_\p\cap S;\\
&\subseteq \p^{khn} V^{-1}S\cap S \hbox{ because }R\setminus \p\subseteq V;\\
&\subseteq \sqrt{\p S}^{khn}V^{-1}S\cap S=\sqrt{\p S}^{(knh)};\\
&\subseteq \sqrt{\p S}^{kn}\hbox{ by Theorems \ref{USP-Poly-CharZero} and \ref{USP-Poly-CharP}};\\
&=(\sqrt{\p S}^{k})^{n}\subseteq \overline{\p S}^{n}\hbox{ by applying Lemma \ref{lemma for descent}};\\
&\subseteq \p^{n-d}S\hbox{ by Uniform Brian\c con-Skoda (Theorem  \ref{ThmUBS}).}
\end{align*}
Then, $\p^{(khn)}\subseteq  \p^{n-d}S\cap R= \p^{n-d}.$

If $k!$ is invertible, the claim follows the same lines as before, but we use the second part of Lemma \ref{lemma for descent}.

\end{proof}

As a corollary of the previous result, we answer a question asked by Takagi to the fourth author.

\begin{corollary}\label{CorTakagi}
Let $S=K[x_1,\ldots,x_d]$, and $G$ a finite group that acts on $S$.
Let $R=S^G$ denote the ring of invariants.
Let $\p\subseteq R$ a prime ideal and $h=\Ht(\p).$
If $k=|G|$ is invertible in $K$, then
$$
\p^{(khn)}\subseteq \p^{n-d}
$$
for every positive integer $n$.
Furthermore, if  $k!$ is invertible in $K$, then
$$
\p^{(khn)}\subseteq \p^{n}.
$$
\end{corollary}

To the best of our knowledge, Question \ref{question} is still open for direct summands $R$ of polynomial rings such that the extension $R\to S$ is infinite. For recent progress on toric rings see the work of Walker \cite{RobertUSTP}. 

%% file: PackingProblem.tex
\section{Symbolic powers of monomial ideals}

\subsection{Symbolic powers, monomial ideals and matroids}

Let $S = K[x_1, \ldots, x_n]$ denote the polynomial ring in $n$ variables over a field $K$, and $\m = \left( x_1, \ldots, x_n \right)$. There is a bijection between the squarefree monomial ideals in $S$ and  simplicial complexes in $n$ vertices, via the Stanley-Reisner correspondence. Some algebraic properties of such an ideal can be described via the combinatorial and topological properties of the corresponding simplicial complex, and vice-versa. Varbaro \cite{Varbaro} and Minh and Trung \cite{TrungMinhCM} have independently shown that the property that all the symbolic powers of a Stanley-Reisner ideal are Cohen-Macaulay is equivalent to a combinatorial condition on the corresponding simplicial complex, namely that the simplicial complex is a matroid.

\begin{definition} {\rm A \it simplicial complex \rm on the set $[n] := \left\lbrace 1, \ldots, n \right\rbrace$ is a collection $\Delta$ of subsets of $[n]$, called \it faces \rm of $\Delta$, that satisfies the following property: given a face $\sigma \in \Delta$, if $\theta \subseteq \sigma$, then $\theta \in \Delta$. A \it facet \rm is a face that is maximal under inclusion.}
\end{definition}  

Given a simplicial complex on $[n]$, we can define a square-free monomial ideal in $S$ corresponding to $\Delta$:

\begin{definition} {\rm Given a simplicial complex $\Delta$, the \it Stanley-Reisner ideal \rm of $\Delta$ is the following square-free monomial ideal:
$$I_{\Delta} = \left( x_{i_1} \cdots x_{i_s} : \left\lbrace i_1, \ldots, i_s \right\rbrace \notin \Delta \right).$$
The quotient $K\left[ \Delta \right] := S / I_{\Delta}$ is called the \it Stanley-Reisner ring \rm of $\Delta$.}
\end{definition}


On the other hand, given a square-free monomial ideal, we can recover the simplicial complex associated to it, giving us a bijective correspondence:

\begin{definition} {\rm Given a square free monomial ideal $I$ in $S$, the \it Stanley-Reisner complex \rm of $I$ is given by
$$\Delta = \left\lbrace \left\lbrace i_1, \ldots, i_s \right\rbrace \subseteq [n] \, | \, x_{i_1} \ldots x_{i_s} \notin I \right\rbrace.$$}
\end{definition}

For a more details about Stanley-Reisner theory, we refer to \cite{SRsurvey,MSAlgComb}.

\begin{definition} {\rm A simplicial complex $\Delta$ on $[n]$ is said to be a \it matroid \rm if for all facets $F, G \in \Delta$ and all $i \in F$, there exists $j \in G$ such that $\left( F \backslash \left\lbrace i \right\rbrace \right) \cup \left\lbrace j \right\rbrace \in \Delta$ is still a facet.}
\end{definition}

This turns out to be precisely the combinatorial condition that corresponds to the following property of the symbolic powers of the Stanley-Reisner ideal:

\begin{theorem}[{Varbaro \cite{Varbaro}, Minh-Trung \cite{TrungMinhCM}\footnote{See also \cite{TrungMinhCMCor}.}}]
\label{ThmCMMatroid}
Given a simplicial complex $\Delta$ on $[n]$, $S / I_{\Delta}^{(m)}$ is Cohen-Macaulay for all $m \geqslant 1$ if and only if $\Delta$ is a matroid.
\end{theorem}

{\cb
We point out that Terai and Trung \cite{TeraiTrung} showed a more general result: If $S / I_{\Delta}^{(m)}$ is Cohen-Macaulay for some $m\geqslant 3,$ then $\Delta$ is a matroid.
}

\begin{example}
The figure below represents the well-known Fano matroid, in $7$ variables, where colinear points correspond to facets, considering the circle as a line. 
The Stanley-Reisner ideal of the Fano matroid in $\mathbb{F}_2 \left[ x_1, \ldots, x_7 \right]$ is given by
$${\cb I = \left( \begin{aligned} x_4 x_2 x_1, x_4 x_3 x_1, x_4 x_3 x_2, x_5 x_2 x_1, x_5 x_3 x_1, x_5 x_3 x_2, x_5 x_4 x_1,\\
      x_5 x_4 x_2, x_6 x_2 x_1, x_6 x_3 x_1, x_6 x_3 x_2, x_6 x_4 x_1, x_6 x_4 x_3, x_6 x_5 x_2,\\
      x_6 x_5 x_3, x_6 x_5 x_4, x_7 x_2 x_1, x_7 x_3 x_1, x_7 x_3 x_2, x_7 x_4 x_2, x_7 x_4 x_3,\\
      x_7 x_5 x_1, x_7 x_5 x_3, x_7 x_5 x_4, x_7 x_6 x_1, x_7 x_6 x_2, x_7 x_6 x_4, x_7 x_6 x_5 \end{aligned} \right).}$$
By Theorem \ref{ThmCMMatroid}, we know that the symbolic powers of $I$ are all Cohen-Macaulay. Using Macaulay2 \cite{M2} we can check that, for example, $I^{(2)} \neq I^2$ and $I^{(3)} \neq I^3$.
%
%
%

\begin{center}
\begin{tikzpicture}
\draw (0,0) node[left]{$x_1$} -- (4,0) node[right]{$x_5$};
\draw (4,0) -- (2,{2*sqrt(3)}) node[above]{$x_3$};
\draw (0,0) -- (2,{2*sqrt(3)});
\draw (2,{2*sqrt(3)/3}) circle ({2*sqrt(3)/3});
\draw (4,0) -- (1,{sqrt(3)}) node[left]{$x_2$};
\draw (0,0) -- (3,{sqrt(3)}) node[right]{$x_4$};
\draw (2,{2*sqrt(3)}) -- (2,0) node[below]{$x_6$};
\draw (2,{2*sqrt(3)/3}) node[right]{$\,\, x_7$};
\end{tikzpicture}
\end{center}

\end{example}

\subsection{The packing problem}
\label{SubsectionPackProb}
\bigskip

Although there are well-known instances in which the symbolic powers of an ideal are equal to
its usual powers, for example complete intersections, there are essentially no theorems which
give necessary and sufficient criteria for this equality to be true, except in a few cases. One of the most notable cases are the prime ideals defining
curves
which are licci \cite{HU}. In this latter case, being a complete intersection is both necessary and
sufficient for the symbolic powers and regular powers to be the same.  There is not even a good guess about what properties of an
ideal are necessary and sufficient to guarantee the equality of the symbolic powers and usual powers.
However, in the case of square-free monomial ideals, there is a beautiful conjecture, first
discovered by Conforti and  Cornu\'ejols \cite{CC} in the context of max-flow min-cut properties, which
was reworked by Gitler, Villarreal and others \cite {GRV},\cite{GVV} to place the conjecture within 
commutative algebra. We present this conjecture, and also introduce a new relative version of it. In addition, we give a proof of this relative version for graph ideals, and for the symbolic square. 
We first need to recall some definitions.

\begin{definition}{\rm  Let $S$ be a polynomial ring over a field. A square-free monomial ideal $I$ of height $c$ is \it K\"onig \rm if there exists a
regular sequence of monomials in $I$ of length $c$.  The ideal $I$ is said to have the \it packing property \rm if every ideal
obtained from $I$ by setting any number of variables equal to $0$ or $1$ is K\"onig.}
\end{definition}

 The conjecture of Conforti and Cornu\'ejols can be restated in this language to say  that the symbolic powers and usual powers of a square-free monomial ideal
coincide if and only if the ideal has the packing property.  This conjecture has been the subject of much scrutiny \cite{FHT,Cornuejols,CMS,CGM,HM,MV,TT}.  If $I$
is the edge ideal of a finite simple graph $G$, this conjecture is known \cite{GVV}; in fact in this case  $I^{(k)}=I^k$ for all
$k$ if and only if  $G$ is bipartite if and only if  $I$ has the packing property. 
In the graph case, the conjecture can also be reintrepreted in terms of well-known graph invariants. Namely, let $G$ be a graph.  Set $c(G)$ equal to the size of minimal vertex cover and $m(G)$ equal to the size of maximal set of disjoint edges. Then the height of $I$ is $c(G)$, and the length of a maximal sequence in $I$ of monomials is $m(G)$. 
Obviously,  $c(G) \geqslant m(G)$. In this language, $G$  has the K\"onig property if and only if $c(G) = m(G)$, while $G$ is packed if
and only if every minor of $G$ has the K\"onig property. 

One direction is not difficult.  Assume that $I^{(n)} = I^n$ for all $n$. One can prove this property is
preserved after setting variables equal to 0 or 1, so to prove that $I$ is packed, one only needs to prove
that $I$ is K\"onig. Moreover, by setting variables which do not appear in the minimal generators of
$I$ to $0$, one can further assume that every variable appears in some minimal prime of $I$. But
if the ring is $K[x_1,...,x_n]$, the monomial $m = x_1\cdots x_n$ is in  $\p^c$ for every minimal
prime $\p$ of $I$, since $I$ has height $c$. It follows that $m\in I^{(c)}$. By assumption, $m\in I^c$, and
this implies that there are monomials $m_1,...m_c\in I$ such that
$m_1\cdots m_c = m$. Since $m$ is square-free these monomials necessarily have disjoint support;  then $m_1,...,m_c$ form a regular sequence in $I$ of maximal length $c$.
The difficult direction is to prove that if $I$ is packed then the symbolic and usual powers agree. 

In this section we introduce a relative version of this conjecture. We make the following definition:

\begin{definition} {\rm Let $I$ be a a square-free monomial ideal. We say that $I$ is $k$-K\"onig if there is a regular sequence of monomials in
$I$ of length at least min$\{k, height(I)\}$. We say $I$ is $k$-packed if every square-free monomial ideal $J$ obtained from $I$ by setting variables
equal to $0$ or $1$ is $k$-K\"onig.}
\end{definition}

With this language, a natural extension of the question of Conforti and   Cornu\'ejols  is:

\begin{question}\label{mainques} Is $I$ $k$-packed if and only if $I^{(n)} = I^n$ for all $n \leqslant k$?
\end{question}

We prove  that if $I$ is the edge ideal of a graph, then Question \ref{mainques}  has a positive answer.  To achieve this, we prove
in our main theorem (see Theorem \ref{mainthm}) that
$I^{(k)}=I^k$ for $1 \leqslant k \leqslant n$ if and only if $G$ contains no odd cycles of length at most $2n-1$. In particular,
if $t$ is chosen to be the least integer such that $I^{(t)}\ne I^t$ (if such a $t$ exists), then $2t-1$ is the size of
the smallest induced cycle of $G$. Another  corollary of our
main theorem is that if 
$G$ is a finite graph with edge ideal  $I$ of height $c$, then  $I^{(k)}=I^k$ for $1 \leqslant k \leqslant c$ implies that  $G$ is bipartite.  We also
observe that Question \ref{mainques} has a positive answer if $k = 2$. We do not know whether
the question has a positive answer if either $I$ is generated by cubics, or for $k = 3$.

We begin our study of the graph ideal case by noting:

\begin{proposition}\label{easydir} Let $G$ be a finite simple graph.
\begin{enumerate}
\item Suppose $G$ has an odd cycle $v_1,\cdots, v_{2n-1}$. Then $f= \prod v_i \in I_G^{(n)} \setminus I_G^n$.
\item Assume that $c(G)>m(G)=n-1$. Then $G$ contains  an odd cycle of length at most $2n-1$. 
\end{enumerate}
\end{proposition}

\begin{proof}$ $
 \begin{enumerate} \item We note that $f\notin I_G^n$ by degree reasons. But any minimal prime of $I_G$ must contain $n$ vertices of the cycle, so $f \in I_G^{(n)}$.
\item Suppose $G$ does not contain such cycle. Since any cycle of length at least $2n+1$ has a set of $n$ disjoint edges,  $G$ must not contain any cycle of odd length. Thus $G$ is bipartite, and K\"onig's theorem asserts that $c(G)= m(G)$, which gives a contradiction. 
\end{enumerate}
\end{proof}

Although our main concern in this section is with the edge ideals of graphs, our first reduction of the problem of
the equality of powers and symbolic powers works equally well for any square-free monomial ideal. We describe this
reduction in the discussion and remark below.

\smallskip
\begin{disc}{\rm  Let $J$ be a square-free monomial ideal, and fix a variable $x$.
We let  $I_x$ denote $J:(x)$ and $I$ be the ideal generated by the monomials in $J$ not involving $x$. We have  $J = I +xI_x$ and $I\subseteq I_x$. 
Suppose
we know that $I_x^{(n)} = I_x^n$ and $I^{(n)}=I^n$ (for example, if $J = I_G$, this would be the case if we know by induction
 that the symbolic powers and usual powers of $I_G$ are the same whenever we set variables equal to $0$ or $1$, provided
we do so for at least one variable).  

Clearly $J = I_x \cap (I,x)$. It follows that

\begin{align*}
J^{(n)} &= I_x^{(n)} \cap (I,x)^{(n)}\\
&= I_x^n\cap (I,x)^n\\
& = I^n  + x(I^{n-1}\cap I_x^n) + x^2(I^{n-2}\cap I_x^n) + \cdots  + x^nI_x^n.
\end{align*}

On the other hand, as $J = I+ xI_x$:

$$J^n  = I^n +  xI^{n-1}I_x + x^2I^{n-2}I_x^2 +\cdots + x^nI_x^n$$
As each term in this expression of $J^n$ is inside the corresponding term of $J^{(n)}$, the equality $J^{(n)} = J^n$ is equivalent to
\begin{equation}\label{EqPP}
 I^k\cap I_x^n = I^kI_x^{n-k} \quad\text{for} \quad 1 \leqslant k \leqslant n-1 
 \end{equation}

We can write $I_x=L+I$ where $L$ is generated by monomials of $J:(x)$ which are
not in $I$ (in the case $J = I_G$, $L$ is the ideal generated by the variables which correspond to neighbors of $x$).

The lefthand side of Equation \ref{EqPP} can be written as:

$$(I^k \cap \sum_{0 \leqslant i \leqslant k-1}I^iL^{n-i})  + \sum^n_{j=k}  I^jL^{n-j}.$$

We summarize this discussion in the following remark:}
\end{disc}

\smallskip
\begin{remark}\label{keyrmk}
{\rm Let $J$ be a square-free monomial ideal, and let $x$ be a variable. 
We let  $I_x$ denote $J:(x)$ and $I$ be the ideal generated by monomials in $J$ not involving $x$.  Suppose
we know that $I_x^{(n)} = I_x^n$ and $I^{(n)}=I^n$. Write $I_x=L+I$ where $L$ is generated by monomials of $J:(x)$ which are
not in $I$. 
Then $J^{(n)} = J^n$  if and only if for all $k$ with $0 \leqslant i < k \leqslant n$
$$ I^k \cap I^iL^{n-i} \subseteq \sum^n_{j=k} I^jL^{n-j}.$$
Note that the right hand side of this expression is precisely $I^kI_x^{n-k}$.}
\end{remark}

Before we begin the proof of  main result in this subsection, we point out that this result also follows from the recent work of Lam and Trung \cite[Corollary 4.5]{LamTrung}.

\begin{theorem}
\label{mainthm} Let $G$ be a finite simple graph, and let $I: = I_G$, the edge ideal of $G$. Then
$I^{(k)}=I^k$ for $1 \leqslant k \leqslant n$ if and only if $G$ contains no odd cycles of length at most $2n-1$. In particular,
if $t$ is chosen to be the least integer such that $I^{(t)}\ne I^t$ (if such a $t$ exists), then $2t-1$ is the size of
the smallest induced cycle of $G$.
\end{theorem}

\begin{proof}  We first prove the last asserted statement, assuming the first.  Let $2s-1$ be the size of the
smallest induced odd cycle. Hence, $G$ has no odd cycles of
length at most $2s-3$, since any odd cycle contains an induced odd cycle of at most the same length. The first
statement of the theorem then proves that $I^{(k)}=I^k$ for $1 \leqslant k \leqslant s-1$. On the other hand,
since $G$ does contain an odd cycle of length $2s-1$, the other direction of the first statement
shows that $I^{(s)}\ne I^{s}$. Hence $t = s$, proving the second statement. We now prove the first
statement.

Proposition \ref{easydir} gives the ``only if" direction of this theorem, so only the ``if" direction needs to be proved. We shall prove this direction  by induction on the number of vertices.  It follows that we may assume that the symbolic and usual powers
agree up to $n$ for any ideal obtained from $I$ by setting variables equal to $0$ or $1$, provided at least one variable is
set equal to $0$ or $1$.  By Remark \ref{keyrmk}, the proof is finished by proving the following lemma (we adopt the
notation from the discussion and remark):

\begin{lemma}
Suppose that $G$ has no odd cycles of length up to $2i+3$. Then
$$I^k \cap I^iL^{n-i} \subseteq \sum^n_{j=k}  I^jL^{n-j} = I^n+I^{n-1}L+ \cdots + I^kL^{n-k}$$
for any $i<k \leqslant n$.
\end{lemma}

\begin{proof}
We use induction on $n$, then a backwards induction on $k$. For the second induction, note that if $k= n$, then the conclusion is satisfied
trivially. 
By way of contradiction, consider any minimal degree monomial  element $f\in I^k\cap I^iL^{n-i}$ such that $f\notin  \sum_k^{n}  I^jL^{n-j}$. 
We may write $f$ as the product of $k$ edges: $x_iy_i$ ($1 \leqslant i \leqslant a$, $x_i \in L$), $u_iv_i$ ($1\leqslant i\leqslant b$, $u_i,v_i \notin L$),  and collect the rest of
the variables dividing $f$ into two sets, a set $z_i$ ($1\leqslant i\leqslant c$, $z_i\in L$), and possibly  some extra variables, none of them in $L$. We denote this extra set of variables by $F$.  Note that there might well be repetition among the variables.  Also observe that no $y_j$ can be in $L$, since this would give
a 3-cycle in $G$. With this notation,
$$f = \prod_{1 \leqslant j \leqslant a}(x_jy_j)\cdot \prod_{1 \leqslant j \leqslant b}(u_jv_j)\cdot \prod_{1\leqslant j\leqslant c} z_j\cdot \prod_{t\in F}t,$$
where $a+b \geqslant k$. We can assume that $a+b = k$, since if the sum is strictly greater, $f$ would be in $I^{k+1}$, and we are done by the decreasing induction on $k$.
We call this expression for $f$ the \it first expression \rm of $f$ (it is  not necessarily unique).

Since by assumption $f\in  I^iL^{n-i}$, we may also write $$f = \prod_{1\leqslant j\leqslant i}m_j\cdot \prod_{1\leqslant j\leqslant l} z_j' \cdot\prod_{w\in W}w,$$
where each $m_j$ is a degree two monomial corresponding to an element in $I$, all $z_j'\in L$, and none of the extra variables $w$ are in $L$. Finally, $l$ is an integer such that $l\geqslant n-i$. We call this expression for $f$ the \it second expression \rm of $f$ (again,
it is not necessarily unique).

We make a general observation: since each $m_j$ divides $f$, if for some $j$,  $m_j= x_ry_r$ or $u_rv_r$ 
 in the first representation of $f$, we can cancel $m_j$ in the first representation so that
$f/m_j\in I^{k-1}\cap I^{i-1}L^{n-i} = I^{n-1} + I^{n-2}L + ... + I^{k-1}L^{n-k}$ by the induction on $n$, and then $f\in \sum_k^{n}  I^jL^{n-j}$, which gives 
a contradiction. Thus, without loss of generality we  assume that $m_j$ is not equal to any $x_ry_r$ or $u_rv_r$.

We claim that $c = 0$, and there are no extra variables $F$.  If not, consider first the variable $z_1$. It must appear among the variables in
the second expression for $f$ as an element of $ I^iL^{n-i}$. If $z_1$ is among the $z_j'$, we can cancel, and use the induction on $n$ to
reach a contradiction.  If not, then $z_1$ divides one of the $m_j$, say $m_1 = z_1s$. As $s$ appears among the variables in the first
expression, we simply combine it with $z_1$ and cancel $m_1$ from both sides. Notice
that $f/m_1\in I^{k-1}$, since the rearrangement affects at most one edge monomial in the first
expression. The induction then gives a contradiction, as we observed above.
  Thus, no $z_i$ can appear in the first expression for $f$. Now consider a
variable $t\in F$. Using the same reasoning, $t$ must appear in the second expression for $f$. It cannot be one of the $z_j'$, by assumption.
If $t$ appears in $W$, we can cancel it, and obtain that $f/t$ is a smaller counterexample, which gives a contradiction. Thus $t$ must divide some $m_j$. 
We can again recombine $t$ in the first expression for $f$ to replace possibly one edge monomial in $I$ by $m_j$. After canceling $m_j$
from each side and using induction, we reach a contradiction.  We have reached the situation in which
$\prod_{1\leqslant j\leqslant i}m_j\cdot \prod_{1\leqslant j\leqslant l} z_j'$ divides $$f = \prod_{1\leqslant j \leqslant a}(x_jy_j)\cdot \prod_{1\leqslant j\leqslant b}(u_jv_j).$$
Since none of the $u_j,v_j,y_j$ are in $L$, without loss of generality, $x_1 = z_1',...,x_{l} = z_{l}'$.  After canceling these elements we
obtain that $\prod_{1 \leqslant j \leqslant i}m_j$ divides $$(y_1y_2\cdots y_{l})\prod_{l+1 \leqslant j\leqslant a} (x_jy_j)\cdot \prod_{1\leqslant j\leqslant b}(u_jv_j).$$ 
Note that $a-l+b = k-l \leqslant k-(n-i) = i + (k-n) < i$ (recall that we may assume $k<n$ now as the base case $k=n$ was handled at the beginning of the proof).  

We next need to do the case $i = 0$ separately. In this case, our assumption is that $G$ has no $3$-cycles.  In this case $l = n$, and since
$x_1 = z_1',...,x_{l} = z_{l}'$, we must have $n$ edge-monomials $x_1y_1,...,x_ny_n$ in the first expression for $f$. This forces $f\in I^n$
(after our reductions), which gives the necessary conclusion. Thus in the remainder of the proof, we may assume that $i\geqslant 1$. An important
consequence of this reduction is that there cannot be an edge between any of the $y_j$, as this would give a $5$-cycle.

We let $D$ be the set of vertices $\{y_1,...,y_{l}\}$. 
By Lemma \ref{oddpath} below applied with $e = a+b-l = k-l <i$ (see above)  and $s_1t_1,...,s_et_e$ being the edges $x_jy_j$  for $l+1 \leqslant j \leqslant a$ together with all the $u_jv_j$, we can find an odd path $w_1,\cdots, w_l$ with $l \leqslant 2e+1$  and $w_1, w_l \in F$ (so they are neighbors of some $x$s). Renumbering if necessary, we have an odd cycle $x, x_1,  w_1, \cdots, w_l, x_2$ of length $l+3$.   Now there can be repetition, but any time we identify two non-adjacent vertices in that cycle, a new odd cycle appears. The proof is finished. \end{proof}

We finish the proof of the main theorem by proving Lemma \ref{oddpath} below.

\begin{lemma}\label{oddpath}
Let $e \geqslant 1$ be an integer. Suppose there are $e$ edges $s_1t_1, \cdots s_et_e$ and a set of vertices $D$, together with an auxilliary set $\{m_1,...m_{e+1}\}$ of
edges (repetitions are allowed in all of these)  such that $\prod_{1 \leqslant j \leqslant e+1} m_j$ divides
$\prod_{1 \leqslant j \leqslant e} (s_jt_j)\cdot \prod_{w\in D}w.$ Assume also that there are no edges between any of the elements in $D$.
 Then there is a path $w_1,\cdots, w_l$ with $l \leqslant 2e+1$ even (so that the path length is odd) and $w_1, w_l \in D$.
\end{lemma}

\begin{proof}
We use induction on $e$.  We claim that two of the edges $m_i$s must contain exactly one from $D$. The reason is simply because there are no edges in $D$, so any $m_i$ that contains some vertices in $D$ contains exactly one, and there has to be at least two such edges, as the product of the $m_i$s has degree $2e+2$, while the degree of the product of $x_jy_j$ is $2e$.

We may write these edges as  $m_1 = w_1u$ and $m_2 = w_2v$ where $w_1,w_2\in D$.
Consider $u$ and $v$. If $uv$ is an edge, then we have a length three path $w_1,u,v,w_2$ and we are done. Note that this settles the case $e=1$, as in that situation $uv$ has to be the edge $s_1t_1$. 

Hence without loss of generality we
may assume that $e>1$,  $u = s_1$ and $v= s_2$.  Now consider the set $D'$ obtained from $D$ by removing $w_1,w_2$ and including $t_1, t_2$, the $e-2$ edges
$s_3t_3,...,s_et_e$, and the set of $e-1$ auxilliary edges $m_3,...,m_{e+1}$.  We claim these sets satisfy the hypothesis of the lemma. We need to show
that there are no edges between the elements of $D'$. If $t_1t_2$ is an edge, then there is a path of length 5 from $w_1$ to $w_2$:
$w_1,s_1,t_1,t_2,s_2,w_2$. If, on the other hand there is a $w_3\in D$, apart from $w_1,w_2$ such that $w_3t_1$ is an edge, then we may
construct a path of length three from $w_1$ to $w_3$, namely $w_1,s_1,t_1,w_3$, and again we are done. Hence  there are
no edges between the vertices in $D'$. 
 By induction there is an odd length path of length at most $2(e-2)+1$ which begins and ends in vertices in $D'$.  Now if either end of this
path  is $t_1$ or $t_2$, we may augment the path by two edges, e.g. by $t_1s_1$, $s_1w_1$, so that the beginning and end of the
path are in $D$.  This adds at most $4$ edges to the path, so that the length, which is still odd,  is at most $2(e-2)+1 + 4 = 2e+1$, as required. \end{proof}

This completes the proof of Theorem \ref{mainthm}.
\end{proof}

\begin{remark}  We thank Susan Morey for pointing out that this result follows by putting together work in \cite{CMS, HM, MV}, although it is not explicitly
stated.  By combining the work in these papers one obtains the statement that a non-bipartite graph $G$ whose smallest odd cycle has length $2t-1$ has the property that the associated primes of the powers of the edge ideal $I$ of $G$ is equal to the minimal primes of $I$ for powers up to
$t-1$, and the $t^{th}$ power of $I$ has an embedded prime.
\end{remark}

\begin{corollary}
Let $G$ be a finite graph and $I$ be its edge ideal. Let $c= \height(I)$. If $I^{(k)}=I^k$ for $1 \leqslant k \leqslant c$ then $G$ is bipartite. 
\end{corollary}

\begin{proof} If $G$ is not bipartite, there is a minimal odd cycle, say $x_1,x_2,....,x_{2i+1}$. The height of the edge ideal of this cycle is $i+1$
since the vertices are distinct by minimality.  Hence $c\geqslant i+1$. By Theorem \ref{mainthm},  $G$ has no odd cycles of length at most
$2c-1$. But $c\geqslant i+1$ implies that $2i+1 \leqslant 2(c-1)+1 = 2c-1$, a we reach a contradiction. \end{proof}



We apply the work from above to give a positive answer to the relative version
of the conjecture of  Conforti and Cornu\'ejols which was given at the beginning of this subsection.  Recall that
Question \ref{mainques} asks if $I$ is $k$-packed if and only if $I^{(n)} = I^n$ for all $n \leqslant k$.
In the case of quadrics, i.e., for edge ideals of simple graphs, we can give a positive answer:

\begin{theorem}\label{k-packed edge ideals} Let $I$ be the edge ideal of a graph, i.e., a square-free ideal generated by quadrics. Then $I$ is $k$-packed for some
$k\geqslant 2$ if and only if $I^{(n)} = I^n$ for all $n \leqslant k$.
\end{theorem}

\begin{proof}  We first prove that if  $I^{(n)} = I^n$ for all $n \leqslant k$ then $I$ is $k$-packed. Since both of these properties are preserved
by setting variables equal to $0$ or $1$, it suffices to prove that $G$ is $k$-K\"onig. Suppose that the height of $I$ is $c$, but that $I$ does
not contain a regular sequence of monomials of length at least the minimum of $k$ and $c$. Let $m$ denote the product of all the variables.
Then $m\in I^{(c)}$, since at every minimal prime $\p$ of $I$, $m\in \p^c$.  If $k\geqslant c$, then by assumption $m\in I^c$. But then $m$ is a product
of $c$ elements of $I$ which are necessarily a regular sequence as $m$ is square-free. This is a contradiction. If $c > k$, then $m\in I^{(c)}\subset
I^{(k)} = I^k$, and again we reach a contradiction. Hence $I$ is $k$-packed. 

Conversely, assume that $I$ is $k$-packed but that $I^{(n)}\ne I^n$ for some $n \leqslant k$. By decreasing $k$ if necessary, we can assume that
$n = k$. By Theorem \ref{mainthm}, $G$ must contain an odd cycle of length at most $2k+1$. Choose a minimal odd cycle in $G$, say $x_1,x_2,...,x_{2j+1}$
of length
$2j+1$, where $j \leqslant k$.
We set all other variables equal to $0$, and apply the assumption that the resulting ideal $J$ is $k$-packed. Moreover, the height of $J$
is at least the height of the edge ideal of this cycle, which has height $j+1$. Therefore there is a regular sequence of length at least $j$ in $J$. This is
clearly impossible since there are only $2j+1$ variables, and such a regular sequence would require at least $2j$ variables.
\end{proof}

There are a few easy cases one can also do, which we record as a remark;

\begin{remark}\label{easycases} {\rm We observe the following: if $I$ is $k$-packed, and $J^{(n)} = J^n$ for all
$n \leqslant k$ and for all ideals $J$ which are obtained from $I$ by setting at least one variable equal to $0$ or $1$, and $I$ has height $k$,
then $I^{(n)} = I^n$ for all
$n \leqslant k$. To see this suppose not and let $f$ be a minimal monomial which is in some $I^{(n)}\setminus I^n$ for some $n \leqslant k$.  We may assume
that $f$ contains every variable, since if not, we can set a variable equal to $0$ without changing the fact that $f$ is in $I^{(n)}\setminus I^n$,
contradicting our assumption. We may also assume that $f$ is square-free. This is because the product of all the variables is clearly in
$\p^k$ for all primes $\p$ containing $I$, since $k$ is the height of $I$.  It follows that $f$ must be the product of all the variables. But since $I$
is $k$-packed there is a regular sequence of monomials in $I$ of length $k$, and the product of them must divide $f$ since they are disjoint.
Thus $f\in I^k$, contradiction.}
\end{remark}

A consequence of the above Remark is that Question \ref{mainques} has a positive answer for $k = 2$. 

\begin{corollary}\label{mermin}
$I$ is $2$-packed if and only if $I^{(2)}=I^2$.
\end{corollary}
\begin{proof}
The case in which $I$ has height one is trivial. If $I$ has height two, the previous remark applies. 

\end{proof}